\documentclass{article}

\usepackage{geometry}                % See geometry.pdf to learn the layout options. There are lots.
\usepackage{enumerate}

\usepackage{graphicx}
\usepackage{epstopdf}

\usepackage{amsmath}
\usepackage{amssymb}
\usepackage{amsthm}
\usepackage{url}
\usepackage{tikz-cd}

\usepackage{pifont}% http://ctan.org/pkg/pifont
\usepackage{stmaryrd}
\usepackage{dsfont}
\usepackage{textcomp}
\usepackage{verbatim}

\newtheorem{thm}{Theorem}[section]
\newtheorem{prop}[thm]{Proposition}
\newtheorem{lemma}[thm]{Lemma}
\newtheorem{cor}[thm]{Corollary}

\theoremstyle{definition}

\newtheorem{dfn}[thm]{Definition}

\newtheorem{constr}[thm]{Construction}
\newtheorem{ax}[thm]{Axioms}
\newtheorem{ntn}[thm]{Notation}
\newtheorem{que}[thm]{Question}
\newtheorem{rem}[thm]{Remark}

\newcommand{\pow}{\mathcal P}

\newcommand{\id}{\mathrm{id}}

\newcommand{\setmany}{\text{\reflectbox{$\mathsf{S}$}}}

\newcommand{\epsin}{\mathrel{\varepsilon}}

\title{Categorical New Foundations}
\author{
Paul Gorbow \\
\\
	\small University of Gothenburg \\
	\small Department of Philosophy, Linguistics, and Theory of Science \\
	\small Box 200, 405 30 G\"OTEBORG, Sweden \\
	\small\url{pgorbow@gmail.com} \\
}
\begin{document}

\maketitle

\section{Introduction}

New Foundations ($\mathrm{NF}$) is a set theory obtained from naive set theory by putting a stratification constraint on the comprehension schema; for example, it proves that there is a universal set $V$, and the natural numbers are implemented in the Fregean way (i.e. $n$ is implemented as the set of all sets with $n$ many elements). $\mathrm{NFU}$ ($\mathrm{NF}$ with atoms) is known to be consistent through its remarkable connection with models of conventional set theory that admit automorphisms. This connection was discovered by Jensen, who established the equiconsistency of $\mathrm{NFU}$ with a weak fragment of $\mathrm{ZF}$, and its consistency with the axiom of choice \cite{Jen69}. (So in the $\mathrm{NF}$-setting atoms matter; Jensen's consistency proof for $\mathrm{NFU}$ does not work for $\mathrm{NF}$.) 

This paper\footnotetext[0]{Thanks to Henrik Forssell, Peter LeFanu Lumsdaine and the anonymous referee for valuable feedback on this paper.}  aims to lay the ground for a category theoretic approach to the study of $\mathrm{NF}$. A first-order theory, $\mathrm{ML}_\mathsf{Cat}$, in the language of categories is introduced and proved to be equiconsistent to $\mathrm{NF}$. $\mathrm{ML}_\mathsf{Cat}$ is intended to capture the categorical content of the predicative version of the class theory $\mathrm{ML}$ of $\mathrm{NF}$. The main result, for which this paper is motivated, is that $\mathrm{NF}$ is interpreted in $\mathrm{ML}_\mathsf{Cat}$ through the categorical semantics. This enables application of category theoretic techniques to meta-mathematical problems about $\mathrm{NF}$-style set theory. Conversely, it is shown that the class theory $\mathrm{ML}$ interprets $\mathrm{ML}_\mathsf{Cat}$, and that a model of $\mathrm{ML}$ can be obtained constructively from a model of $\mathrm{NF}$. Each of the results in this paragraph is shown for the versions of the theories with and without atoms, both for intuitionistic and classical logic.\footnote{Due to the lack of knowledge about the consistency strength of $\mathrm{INF(U)}$, the non-triviality of the statement $\mathrm{Con(INF(U)) \Rightarrow Con(IML(U))}$ needs to be taken as conditional, see Remark \ref{remML}.} Therefore, we use the notation $\mathrm{(I)NF(U)}$ and $\mathrm{(I)ML(U)}$, where the $\mathrm{I}$ stands for the intuitionistic version and the $\mathrm{U}$ stands for the version with atoms, so four versions of the theories are considered in parallel. An immediate corollary of these results is that $\mathrm{(I)NF}$ is equiconsistent to $\mathrm{(I)NFU} + |V| = |\mathcal{P}(V)|$. For the classical case, this has already been proved in \cite{Cra00}, but the intuitionistic case appears to be new. Moreover, the result becomes quite transparent in the categorical setting. 

Just like a category of classes has a distinguished subcategory of small morphisms (cf. \cite{ABSS14}), a category modelling $\mathrm{(I)ML(U)}_\mathsf{Cat}$ has a distinguished subcategory of type-level morphisms. This corresponds to the distinction between sets and proper classes in $\mathrm{(I)NF(U)}_\mathsf{Set}$. With this in place, the axiom of power objects familiar from topos theory can be appropriately reformulated for $\mathrm{(I)ML(U)}_\mathsf{Cat}$. It turns out that the subcategory of type-level morphisms contains a topos as a natural subcategory.

Section \ref{stratAlgMot} explains the motivation behind the paper and in Section \ref{stratAlgBack} some further background is given. 

Section \ref{Stratified set theory and class theory} introduces the set theories $\mathrm{(I)NF(U)}_\mathsf{Set}$ and the class theories $\mathrm{(I)ML(U)}_\mathsf{Class}$. Here we also establish that $\mathrm{NF(U)}_\mathsf{Set}$ is equiconsistent to $\mathrm{ML(U)}_\mathsf{Class}$, through classical model theory.

In Section \ref{CatAxioms}, categorical semantics is explained in the context of Heyting and Boolean categories. This semantics is applied to show generally that $\mathrm{(I)NF(U)}_\mathsf{Set}$ is equiconsistent to $\mathrm{(I)ML(U)}_\mathsf{Class}$.

The axioms of the novel categorical theory $\mathrm{(I)ML(U)}_\mathsf{Cat}$ are given in Section \ref{Formulation of ML_CAT}, along with an interpretation of $\mathrm{(I)ML(U)}_\mathsf{Cat}$ in $\mathrm{(I)ML(U)}_\mathsf{Class}$. 

It is only after this that the main original results are proved. Most importantly, in Section \ref{interpret set in cat}, category theoretic reasoning is used to validate the axioms of $\mathrm{(I)NF(U)}_\mathsf{Set}$ in the internal language of $\mathrm{(I)ML(U)}_\mathsf{Cat}$ through the categorical semantics. This means that $\mathrm{(I)NF(U)}_\mathsf{Set}$ is interpretable in $\mathrm{(I)ML(U)}_\mathsf{Cat}$. The equiconsistency of $\mathrm{(I)NF}_\mathsf{Set}$ and $\mathrm{(I)NFU} + |V| = |\mathcal{P}(V)|$ is obtained as a corollary. 

In Section \ref{subtopos}, it is shown that every $\mathrm{(I)ML(U)}_\mathsf{Cat}$-category contains a topos as a subcategory. 

In Section \ref{where to go NF cat}, we discuss possible directions for further research.

\section{Motivation}\label{stratAlgMot}

$\mathrm{NF}$ corresponds closely with the simple theory of types, TST, an extensional version of higher order logic which Chwistek and Ramsey indepentently formulated as a simplification of Russell and Whitehead's system in Principia Mathematica. It was from contemplation of TST that Quine introduced $\mathrm{NF}$ \cite{Qui37}. Essentially, $\mathrm{NF}$ is obtained from TST by forgetting the typing of the relations while retaining the restriction on comprehension induced by the typing (thus avoiding Russell's paradox). This results in the notion of stratification, see Definition \ref{DefStrat} below. Thus $\mathrm{NF}$ and $\mathrm{NFU}$ resolve an aspect of type theory which may be considered philosophically dissatisfying: Ontologically, it is quite reasonable to suppose that there are relations which can take both individuals and relations as relata. The simplest example is probably the relation of identity. But in type theory, it is not possible to relate entities of different types. We cannot even say that they are unequal. Since the universe of $\mathrm{NF}$ or $\mathrm{NFU}$ is untyped, such issues disappear. It is therefore not surprising that stratified set theory has attracted attention from philosophers. For example, Cocchiarella applied these ideas to repair Frege's system \cite{Coc85} (a similar result is obtained in \cite{Hol15}), and Cantini applied them to obtain an interesting type-free theory of truth \cite{Can15}. Along a similar line of thought, the categorical version of $\mathrm{NF}$ and $\mathrm{NFU}$ brought forth in this paper may well be helpful for transferring the ideas of stratified set theory to research in formal ontology. In a formal ontology, one may account for what individuals, properties, relations and tropes exist, where properties and relations are considered in an intensional rather than an extensional sense. Roughly, in category theory the objects are non-extensional, but the morphisms are extensional, and this is arguably fitting to the needs of formal ontology.

$\mathrm{NFU}$ is also intimitely connected with the highly active research on non-standard models of arithmetic and set theory, as is further explained in the next section. Out of this connection, Feferman proposed a version of $\mathrm{NFU}$ as a foundation for category theory, allowing for such unlimited categories as the category of all sets, the category of all groups, the category of all topological spaces, the category of all categories, etc \cite{Fef06}. This line of research was further pursued by Enayat, McKenzie and the author in \cite{EGM17}. In short, conventional category theory works perfectly fine in a subdomain of the $\mathrm{NFU}$-universe, but the unlimited categories live outside of this subdomain, and their category theoretic properties are unconventional. Even though they are unconventional (usually failing to be cartesian closed), one might argue that nothing is lost by including them in our mathematical universe. These categories remain to be systematically studied.

The need for a categorical understanding of stratified set theory is especially pressing since very little work has been done in this direction. McLarty showed in \cite{McL92} that the ``category of sets and functions'' in $\mathrm{NF}$ is not cartesian closed. However, several positive results concerning this category were proved in an unpublished paper by Forster, Lewicki and Vidrine \cite{FLV14}: In particular, they showed that it has a property they call ``pseudo-Cartesian closedness''. Similarly, Thomas showed in \cite{Tho17} that it has a property he calls ``stratified Cartesian closedness''. The moral is that it is straight forward to show in $\mathrm{INFU}$, that if $A$ and $B$ are sets, which are respectively isomorphic to sets of singletons $A'$ and $B'$, then the set of functions from $\bigcup A'$ to $\bigcup B'$ is an exponential object of $A$ and $B$. ($V$ is not isomorphic to any set of singletons.) In \cite{FLV14} a generalization of the notion of topos was proposed, with ``the category of sets and functions'' of $\mathrm{NF}$ as an instance. It has however not been proved that the appropriate extension $T$ of this theory (which $\mathrm{NF}$ interprets) satisfies $\mathrm{Con}(T) \Rightarrow \mathrm{Con}(\mathrm{NF})$. Using the results of Section \ref{interpret set in cat} of this paper, it seems within reach to obtain that result by canonically extending a model of $T$ to a model of the categorical theory $\mathrm{ML}_\mathsf{Cat}$ introduced here. That line of research would also help carve out exactly what axioms of $T$ are necessary for that result. Moreover, in \cite{FLV14} it was conjectured that any model of $T$ has a subcategory which is a topos. In Section \ref{subtopos} of this paper, it is proved that every model of $\mathrm{(I)ML(U)}_\mathsf{Cat}$ has a subcategory which is a topos.

A related direction of research opened up by the present paper is to generalize the techniques of automorphisms of models of conventional set theory, in order to study automorphisms of topoi. The author expects that a rich landscape of models of $\mathrm{(I)ML(U)}_\mathsf{Cat}$ would be uncovered from such an enterprise. For example, just like there is a topos in which every function on the reals is continuous, a similar result may be obtainable for $\mathrm{IMLU}_\mathsf{Cat}$ by finding such a topos with an appropriate automorphism. Given the intriguing prospects for founding category theory in stratified set theory, this would open up interesting possibilities for stratified category theoretic foundation of mathematics.

The categorical approach to $\mathrm{NF}$ is also promising for helping the metamathematical study of $\mathrm{NF}$. As stated in the introduction, the main result of this paper has the immediate corollary that $\mathrm{(I)NF}$ is equiconsistent to $\mathrm{(I)NFU} + |V| = |\mathcal{P}(V)|$. A major open question in the metamathematics of $\mathrm{NF}$ is whether $\mathrm{NF}$ (or even its intuitionistic counterpart, which has not been shown to be equiconsistent to $\mathrm{NF}$) is consistent relative to a system of conventional set theory (independent proof attempts by Gabbay and Holmes have recently been put forth). So yet a motivation for introducing $\mathrm{ML}_\mathsf{Cat}$ is simply that the flexibility of category theory may make it easier to construct models of $\mathrm{ML}_\mathsf{Cat}$, than to construct models of $\mathrm{NF}$, thus aiding efforts to prove and/or simplify proofs of $\mathrm{Con(NF)}$.

Since categorical model theory tends to be richer in the intuitionistic setting, an intriguing line of research is to investigate the possibilities for stratified dependend type theory. Dependent type theory is commonly formulated with a hierarchy of universes. In a sense, this hierarchy is inelegant and seemingly redundant, since any proof on one level of the hierarchy can be shifted to a proof on other levels of the hierarchy. Model-theoretically, this can be captured in a model with an automorphism. Since the semantics of type theory tends to be naturally cast in category theory, the understanding arising from the present paper would be helpful in such an effort.

In conclusion, ``categorification'' tends to open up new possibilities, as forcefully shown by the fruitfulness of topos theory as a generalization of set theory. In the present paper it has already resulted in a simple intuitive proof of the old result of Crabb\'e stated above. So given the relevance of $\mathrm{NF}$ and $\mathrm{NFU}$ to type theory, philosophy, non-standard models of conventional set theory and the foundations of category theory, it is important to investigate how $\mathrm{NF}$ and $\mathrm{NFU}$ can be expressed as theories in the language of category theory.

\section{Background}\label{stratAlgBack}

Algebraic set theory, categorical semantics and categorical logic more generally, have been developed by a large number of researchers. An early pioneering paper of categical logic is \cite{Law63}; Joyal and Moerdijk started out algebraic set theory with \cite{JM91} and wrote a short book on the subject \cite{JM95}. The present paper is largely influenced by the comprehensive work of Awodey, Butz, Simpson and Streicher embodied in \cite{ABSS14}. It parallels its approach to an algabraic set theory of categories of classes, but for the $\mathrm{NF}$ context. A category of classes $\mathbf{C}$ is a (Heyting) Boolean category with a subcategory $\mathbf{S}$, satisfying various axioms capturing the notion of `smallness', and with a universal object $U$, such that every object is a subobject of $U$. While the axiomatization of categories of classes naturally focuses on the notion of smallness, the axiomatization in this paper focuses on the notion of type level stratification. Like in \cite{ABSS14}, a restricted notion of power object is obtained, which facilitates interpretation of set theory in the categorical semantics, but the reformulation of the power object axiom needed for the stratified setting is quite different.

$\mathrm{NF}$ may be axiomatized as Extentionality + Stratified Comprehension ($\mathrm{Ext} + \mathrm{SC}$); it avoids the Russell paradox of naive set theory, by restricting comprehension to stratified formulae, a notion to be defined below.  An introduction to $\mathrm{NF}$ is given in \cite{For95}. For any basic claims about $\mathrm{NF}$ in this paper, we implicitly refer to that monograph.

Though the problem of proving the consistency of $\mathrm{NF}$ in terms of a traditional $\mathrm{ZF}$-style set theory turned out to be difficult, Jensen proved the consistency of the subsystem $\mathrm{NFU}$, where $\mathrm{Ext}$ is weakened to $\mathrm{Ext'}$ ($\mathrm{Ext'}$ is extentionality for non-empty sets, thus inviting atoms), in \cite{Jen69}. Jensen used Ramsey's theorem to obtain a particular  model of Mac Lane set theory with an automorphism, and it is relatively straight forward to obtain a model of $\mathrm{NFU} + \textnormal{Infinity} + \textnormal{Choice}$ from that model. There are various interesting axioms that can be added to $\mathrm{NFU}$ to increase its consistency strength. As the understanding of automorphisms of non-standard models of $\mathrm{ZF}$-style set theories has increased, several results on the consistency strength of such extensions of $\mathrm{NFU}$ have been obtained in the work of Solovay \cite{Sol97}, Enayat \cite{Ena04} and McKenzie \cite{McK15}. In this paper we define $\mathrm{NFU}$ as a theory in the language $\{\in, S, \langle -, -\rangle\}$, where $S$ is a predicate distinguishing sets from atoms and $\langle -, -\rangle$ is a primitive pairing function, axiomatized as $\mathrm{Ext}_S + \mathrm{SC}_S + \mathrm{P} + \mathrm{Sethood}$, where $\mathrm{Ext}_S$ and $\mathrm{SC}_S$ are extensionality and stratified comprehension for sets, P regulates $\langle -, -\rangle$ and Sethood regulates $S$, to be specified below. $\mathrm{NFU}$ proves Infinity and is equiconsistent with Mac Lane set theory; $\mathrm{Ext}_S + \mathrm{SC}_S + \mathrm{Sethood}$ is weaker and does not prove Infinity. In this paper, $\mathrm{NF}$ is defined as $\mathrm{NFU}$ + ``everything is a set'', which (in classical logic) is equivalent to $\mathrm{Ext} + \mathrm{SC}$. An introduction to $\mathrm{NFU}$ and extended systems is given in \cite{Hol98}. For any basic claims about $\mathrm{NFU}$ in this paper, we implicitly refer to that monograph.

The theories $\mathrm{NF}$ and $\mathrm{NFU}$ in intuitionistic logic will be referred to as $\mathrm{INF}$ and $\mathrm{INFU}$, respectively. Note that the way $\mathrm{NFU}$ and $\mathrm{NF}$ are axiomatized in this paper, the intuitionistic versions $\mathrm{INFU}$ and $\mathrm{INF}$ also satisfy e.g. the axiom of ordered pair. But if $\mathrm{INF}$ were axiomatized as $\mathrm{Ext} + \mathrm{SC}$ with intuitionistic logic, as done e.g. in \cite{Dzi95}, it is not clear that the resulting intuitionsitic theory would be as strong.

$\mathrm{NF}$ and $\mathrm{NFU}$ also have finite axiomatizations, as shown in \cite{Hai44}, which clarify that their ``categories of sets and functions'' are Boolean categories. In this paper certain extentions of the theories of (Heyting) Boolean categories (in the language of category theory) are proved equiconsistent to $\mathrm{(I)NF(U)}$, respectively.

\section{Stratified set theory and class theory}\label{Stratified set theory and class theory}

Let $\mathcal{L_\mathsf{Set}} = \{\in, S, \langle -, - \rangle\}$ be the language of set theory augmented with a unary predicate symbol $S$ of ``sethood\hspace{1pt}'' and a binary function symbol $\langle - , - \rangle$ of ``ordered pair''. We introduce notation for the ``set-many quantifier'': 
$$\setmany z . \phi \text{ abbreviates } \exists x . \big( S(x) \wedge \forall z . (z \in x \leftrightarrow \phi(z)) \big),$$
where $x$ is chosen fresh, i.e. not free in $\phi$.

\begin{dfn}\label{DefStrat}
Let $\phi$ be an $\mathcal{L_\mathsf{Set}}$-formula. $\phi$ is {\em stratified} if there is a function $s : \mathrm{term}(\phi) \rightarrow \mathbb{N}$, where $\mathrm{term}(\phi)$ is the set of terms occurring in $\phi$, such that for any $u, v, w \in \mathrm{term(\phi)}$ and any atomic subformula $\theta$ of $\phi$,
\begin{enumerate}[(i)]
\item if $u \equiv \langle v, w \rangle$, then $s(u) = s(v) = s(w)$,
\item if $\theta \equiv (u = v)$, then $s(u) = s(v)$,
\item if $\theta \equiv (u \in v)$, then $s(u) + 1 = s(v)$,
\end{enumerate}
where $\equiv$ denotes literal equality (of terms or formulae). Such an $s$ is called a {\em stratification} of $\phi$. $s(u)$ is called the {\em type} of $u$. Clearly, if $\phi$ is stratified, then there is a {\em minimal} stratification in the sense that $s(v) = 0$ for some variable $v$ occurring in $\phi$. Also note that the formula $\langle v, w \rangle = \{\{v\}, \{v, w\}\}$, stipulating that the ordered pair is the Kuratowski ordered pair, is not stratified. Therefore, it is condition (i), read as ``type-level ordered pair'', that gives power to axiom P below. 
\end{dfn}

\begin{ntn}
In the axiomatizations below, $\mathrm{NFU}_\mathsf{Set}$ is the theory thus axiomatized in classical logic, while $\mathrm{INFU}_\mathsf{Set}$ is the theory thus axiomatized in intuitionistic logic. For brevity we simply write $\mathrm{(I)NFU}_\mathsf{Set}$, and similarly for $\mathrm{(I)NF}_\mathsf{Set}$, to talk about the intuitionistic and classical theories in parallel. More generally, any statement that $\mathrm{(I)XX(U)}_\mathrm{K}$ relates to $\mathrm{(I)YY(U)}_\mathrm{L}$ in some way, means that each of the four theories $\mathrm{IXXU}_\mathrm{K}$, $\mathrm{XXU}_\mathrm{K}$, $\mathrm{IXX}_\mathrm{K}$, $\mathrm{XX}_\mathrm{K}$ relates in that way to $\mathrm{IXXU}_\mathrm{L}$, $\mathrm{XXU}_\mathrm{L}$, $\mathrm{IXX}_\mathrm{L}$, $\mathrm{XX}_\mathrm{L}$, respectively. Since we will be proving equiconsistency results between theories in different languages, the language is emphasized as a subscript to the name of the theory. This is why we write $\mathrm{(I)NF(U)}_\mathsf{Set}$ for the set theoretic theory $\mathrm{(I)NF(U)}$. 
\end{ntn}

\begin{ax}[$\mathrm{(I)NFU}_\mathsf{Set}$] \label{SCAx}
\[
\begin{array}{rl}
\mathrm{Ext}_S & (S(x) \wedge S(y) \wedge \forall z . z \in x \leftrightarrow z \in y) \rightarrow x = y \\
\mathrm{SC}_S & \text{For all stratified $\phi$: } \setmany z . \phi(z) \\
\mathrm{P} & \langle x, y \rangle = \langle x\hspace{1pt}', y\hspace{1pt}' \rangle \rightarrow ( x = x\hspace{1pt}' \wedge y = y\hspace{1pt}' ) \\
\textnormal{{ Sethood}} & z \in x \rightarrow S(x) \\
\end{array}
\]
\end{ax}
$\mathrm{Ext}_S$ stands for Extensionality (for Sets), $\mathrm{SC}_S$ stands for Stratified Comprehension (yielding Sets), and $\mathrm{P}$ stands for Ordered Pair. In order to keep the treatment uniform, we axiomatize $\mathrm{(I)NF}_\mathsf{Set}$ as $\mathrm{(I)NFU}_\mathsf{Set}$ + $\forall x . S(x)$. Obviously, $\mathrm{(I)NF}_\mathsf{Set}$ can be axiomatized in the language without the predicate $S$, simply as $\mathrm{Ext} + \mathrm{SC} + \mathrm{P}$ (where $\mathrm{Ext}$ and $\mathrm{SC}$ are like $\mathrm{Ext}_S$ and $\mathrm{SC}_S$, respectively, but without the $S$-conjuncts). Less obviously, $\mathrm{NF}$ proves the negation of Choice \cite{Spe53}, which entails the axiom of Infinity, which in turn enables implementation of type-level ordered pairs. So $\mathrm{NF}$ can be axiomatized as $\mathrm{Ext} + \mathrm{SC}$ in the plain language $\{\in\}$ of set theory.

Note that $\mathrm{SC}_S$ implies the existence of a universal set, denoted $V$. In the context of the sethood predicate, it is natural to restrict the definition of subset to sets. So define
$$x \subseteq y \Leftrightarrow_{\mathrm{df}} S(x) \wedge S(y) \wedge \forall z . (z \in x \rightarrow z \in y).$$
The power set, $\pow y$, of $y$ is defined as $\{z \mid z \subseteq y\}$, and exists by $\mathrm{SC}_S$. Therefore, only sets are elements of power sets. An important special case of this is that $S(x) \leftrightarrow x \in \pow V$. So the axiom $\forall x . S(x)$, yielding $\mathrm{(I)NF}$, may alternatively be written $V = \pow V$. In the meta-theory $\subseteq$ and $\pow$ are defined in the standard way. When proving the existence of functions (coded as sets of ordered pairs) in $\mathrm{(I)NF(U)}$, the type-level requirement of ordered pairs means that the defining formula (in addition to being stratified) needs to have the argument- and value-variable at the same type.

$\mathrm{(I)ML(U)}_\mathsf{Class}$ is the impredicative theory of classes corresponding to $\mathrm{(I)NF(U)}_\mathsf{Set}$. $\mathrm{ML}$ was introduced by Quine in his book \cite{Qui40}. Apparently $\mathrm{ML}$ stands for ``Mathematical Logic'' (the title of that book). There is both a predicative and an impredicative version of $\mathrm{ML}$, and both are equiconsistent with $\mathrm{NF}_\mathsf{Set}$, as proved in \cite{Wan50}. One obtains a model of $\mathrm{ML}$ simply by taking the power set of a model of $\mathrm{NF}$, along with a natural interpretation that suggests itself, so the proof requires enough strength in the meta-theory to handle sets of the size of the continuum. (The equiconsistency between predicative $\mathrm{ML}$ and $\mathrm{NF}$ can be proved in a weaker meta-theory that is only strong enough to handle countable sets.) Without difficulty, the proof extends to equiconsistency between each of the theories $\mathrm{(I)ML(U)}_\mathsf{Class}$ and $\mathrm{(I)NF(U)}_\mathsf{Set}$, respectively. For the purpose of completeness, a proof of $\mathrm{Con}(\mathrm{(I)NF(U)}_\mathsf{Set}) \Rightarrow  \mathrm{Con}(\mathrm{(I)ML(U)}_\mathsf{Class})$ is provided below.

The theory $\mathrm{(I)ML(U)}_\mathsf{Cat}$, which the author introduces in this research as a category theory of $\mathrm{(I)NF(U)}$, probably corresponds better to predicative $\mathrm{(I)ML(U)}$. The difficult and interesting direction of the proof of equiconsistency between $\mathrm{(I)ML(U)}_\mathsf{Cat}$ and $\mathrm{(I)NF(U)}_\mathsf{Set}$ is the interpretation of $\mathrm{(I)NF(U)}_\mathsf{Set}$ in $\mathrm{(I)ML(U)}_\mathsf{Cat}$.

We axiomatize $\mathrm{(I)ML(U)}_\mathsf{Class}$ in a one-sorted language $\mathcal{L}_\mathsf{Class}$ that augments $\mathcal{L}_\mathsf{Set}$ with a unary predicate $C$ and a unary predicate $\mathrm{Setom}$. We read $C(x)$ as ``$x$ is a {\em class}'' and read $S(x)$ as ``$x$ is a {\em set}''.  Moreover, ``Setom'' is a portmanteau for ``sets and atoms''. $\mathrm{Setom}(x) \wedge \neg S(x)$ is read as ``$x$ is an {\em atom}''. We treat the pairing function as a partial function; formally we take it to be a ternary relation symbol, but we write it in functional notation. For convenience, we introduce the abbreviations $\exists \vec{x} \in \mathrm{Setom} . \phi$ and $\forall \vec{x} \in \mathrm{Setom} . \phi$ for $\exists \vec{x} . ((\mathrm{Setom}(x_1) \wedge \dots \wedge \mathrm{Setom}(x_n)) \wedge \phi)$ and $\forall \vec{x} . ((\mathrm{Setom}(x_1) \wedge \dots \wedge \mathrm{Setom}(x_n)) \rightarrow \phi)$, respectively, where $\vec{x} = (x_1, \dots, x_n)$ for some $n \in \mathbb{N}$. We say that such quantifiers are {\em bounded} to $\mathrm{Setom}$.

\begin{ax}[$\mathrm{(I)MLU}_\mathsf{Class}$] \label{AxIMLU}
\[
\begin{array}{rl}
\textnormal{C-hood} & z \in x \rightarrow C(x) \\
\textnormal{Sm-hood} & z \in x \rightarrow \mathrm{Setom}(z) \\
\mathrm{Ext}_C & (C(x) \wedge C(y) \wedge \forall z . (z \in x \leftrightarrow z \in y)) \rightarrow x = y \\
\mathrm{CC}_C & \text{For all $\phi$: }  \exists x . \big( C(x) \wedge \forall z \in \mathrm{Setom} . (z \in x \leftrightarrow \phi(z))\big) \\
\mathrm{SC}_S & \text{For all stratified $\phi$ with only $z, \vec{y}$ free: } \\
& \forall \vec{y} \in \mathrm{Setom} . \exists x \in \mathrm{Setom} . \forall z \in \mathrm{Setom} . (z \in x \leftrightarrow \phi(z, \vec{y})) \\
\textnormal{P} & \forall x, y, x\hspace{1pt}', y\hspace{1pt}' \in \mathrm{Setom} . \\ 
& ( \langle x, y \rangle = \langle x\hspace{1pt}', y\hspace{1pt}' \rangle \leftrightarrow ( x = x\hspace{1pt}' \wedge y = y\hspace{1pt}' ) ) \\
S = \mathrm{Sm} \cap C & S(x) \leftrightarrow (\mathrm{Setom}(x) \wedge C(x))
\end{array}
\]
\end{ax}
C-hood stands for Classhood, S-hood stands for Setomhood, $\mathrm{Ext}_C$ stands for Extensionality (for classes), $\mathrm{CC}_C$ stands for Class Comprehension (yielding classes), and $S = \mathrm{Sm} \cap C$ stands for Set equals Setom Class. In $\mathrm{CC}_C$ and $\mathrm{SC}_S$, we assume that $x$ is fresh, i.e. not free in $\phi$. We obtain $\mathrm{(I)ML}_\mathsf{Class}$ by adding the axiom $\forall x \in \mathrm{Setom} . S(x)$. Predicative $\mathrm{(I)ML(U)}_\mathsf{Class}$ is obtained by requiring in $\mathrm{CC}_C$ that all quantifiers in $\phi$ are bounded to $\mathrm{Setom}$.

The leftwards arrow has been added to the Ordered Pair axiom, because the partial function of ordered pair is formally treated as a ternary relation symbol. One might find it natural to add the axiom $\neg C(x) \rightarrow \mathrm{Setom}(x)$, but since we will not need it, the author prefers to keep the axiomatization more general and less complicated.

The extension of $\mathrm{Setom}$ may be thought of as the collection of sets and atoms, but although $\forall x \in\mathrm{Setom} . (S(x) \vee \neg S(x))$ follows from the law of excluded middle in $\mathrm{MLU}_\mathsf{Class}$, this proof does not go through intuitionistically; the author does not expect it to be provable in $\mathrm{IMLU}_\mathsf{Class}$. Note that it follows from the axioms that Sethood (restricted to Setoms) holds, i.e. that $\forall x \in \mathrm{Setom} . ( z \in x \rightarrow S(x) )$. 

The predicate $S$ is clearly redundant in the sense that it is definable, but it is convenient to have it in the language. This more detailed presentation is chosen because it makes it easy to see that $\mathrm{(I)ML(U)}_\mathsf{Class}$ interprets $\mathrm{(I)NF(U)}_\mathsf{Set}$: For any axiom of $\mathrm{(I)NF(U)}_\mathsf{Set}$, simply interpret it as the formula obtained by replacing each subformula of the form $\boxminus x . \phi$ by $\boxminus x \in \mathrm{Setom} . \phi$, for each $\boxminus \in \{\exists, \forall\}$. One may also obtain a model of $\mathrm{(I)NF(U)}_\mathsf{Set}$ from a model of $\mathrm{(I)ML(U)}_\mathsf{Class}$, by restricting its domain to the extension of $\mathrm{Setom}$ and then taking the reduct to $\mathcal{L}_\mathsf{Set}$. 

We now proceed towards showing that the consistency of $\mathrm{(I)NF(U)}_\mathsf{Set}$ implies the consistency of $\mathrm{(I)ML(U)}_\mathsf{Class}$. The idea of the proof is straightforward: we start with a model of $\mathrm{(I)NF(U)}_\mathsf{Set}$ and add all the possible subsets of this structure as new elements to model the classes, with the obvious extension of the $\in$-relation. However, the proof involves some detail of presentation, especially if we do it directly for intuitionistic Kripke models. So here we start off with the classical case, showing how to construct a model of $\mathrm{ML(U)}_\mathsf{Class}$ from a model of $\mathrm{NF(U)}_\mathsf{Set}$. After the categorical semantics has been introduced, we will be able to perform the same proof in the categorical semantics of any topos (Theorem \ref{ToposModelOfML}). The proof below is therefore redundant, but it may help the reader unfamiliar with categorical semantics to compare the two.

\begin{prop}\label{ConSetConClass prop}
If there is a model of $\mathrm{NF(U)}_\mathsf{Set}$, then there is a model of $\mathrm{ML(U)}_\mathsf{Class}$.
\end{prop}
\begin{proof}
We concentrate on the case $\mathrm{Con(NFU)}_\mathsf{Set} \Rightarrow \mathrm{Con(MLU)}_\mathsf{Class}$. Afterwards it will be easy to see the modifications required for the other case. We take care to do this proof in intuitionistic logic, as it will be relevant later on.

Let $\mathcal{N} = ( N, S^\mathcal{N}, \in^\mathcal{N}, P^\mathcal{N} )$ be a model of $\mathrm{NFU}$. Define a model 
$$\mathcal{M} = ( M , C^\mathcal{N} , \mathrm{Setom}^\mathcal{M} , S^\mathcal{N} , \in^\mathcal{M} , P^\mathcal{N}   )$$ 
as follows. Since $\mathcal{N} \models \mathrm{Ext}_S$, it is straightforward to construct a set $M$ with an injection $p : \mathcal{P}(N) \rightarrow M$ and an injection $t : N \rightarrow M$, such that 
$$\forall x \in N . \forall y \in \mathcal{P}(N) . \big( t(x) = p(y) \leftrightarrow (x \in S^\mathcal{N} \wedge y = \{u \in N \mid u \in^\mathcal{N} x \} ) \big).$$ 
Take $M$ as the domain of $\mathcal{M}$.
$$
\begin{array}{rcl}
C^\mathcal{M}  & =_\mathrm{df} & \{p(y) \mid y \in \mathcal{P}(N ) \} \\
\mathrm{Setom}^\mathcal{M}  & =_\mathrm{df} & \{t (x) \mid x \in N \}  \\
S^\mathcal{M}  & =_\mathrm{df} & \{t (x) \mid x \in S^\mathcal{N} \} \\
u \in^\mathcal{M} v & \Leftrightarrow_\mathrm{df} & \exists x \in N . \exists y \in \mathcal{P}(N) . (u = t(x) \wedge v = p(y) \wedge x \in y) \\ 
P^\mathcal{M}  & =_\mathrm{df} & \{\langle t (x), t (y), t (z) \rangle \mid P^\mathcal{N} (x, y) = z\} \\
\end{array}
$$

We now proceed to verify that $\mathcal{M} \models \mathrm{MLU}_\mathsf{Class}$.

Classhood follows from the construction of $C^\mathcal{M}$ and $\in^\mathcal{M} $.

Setomhood follows from the construction of $\mathrm{Setom}^\mathcal{M} $ and $\in^\mathcal{M} $.

Note that $t $ witnesses 
$$\langle N , \in^\mathcal{N} , S^\mathcal{N} , P^\mathcal{N}  \rangle \cong \langle \mathrm{Setom}^\mathcal{M} , \in^\mathcal{M}  \restriction_{\mathrm{Setom}^\mathcal{M} }, S^\mathcal{M} , P^\mathcal{M}  \rangle.$$ 
For by construction, it is easily seen that it is a bijection and that the isomorphism conditions for $S$ and $P$ are satisfied. Moreover, for any $x, x\hspace{1pt}' \in N $, we have 
\[
\begin{array}{rl}
& t (x) \in^\mathcal{M}  t (x\hspace{1pt}') \\ 
\Leftrightarrow & \exists y \in \mathcal{P}(N) . (t(x\hspace{1pt}') = p(y) \wedge x \in y) \\
\Leftrightarrow & \exists y \in \mathcal{P}(N) . ( x\hspace{1pt}' \in S^\mathcal{N} \wedge y = \{u \in N \mid u \in^\mathcal{N} x\hspace{1pt}' \} \wedge x \in y) \\
\Leftrightarrow & x \in^\mathcal{N}  x\hspace{1pt}'.
\end{array}
\]

Since the axioms $\mathrm{SC}_S$ and Ordered Pair in effect have all quantifiers restricted to the extension of $\mathrm{Setom}$, and $\mathcal{N}$ satisfies these axioms, the isomorphism $t$ yields that $\mathcal{M}$ satisfies these axioms as well.

$\mathrm{Ext}_C$ follows from that $t$ is injective and $\forall x \in N . \forall y \in \mathcal{P}(N) . (t(x) \in^\mathcal{M} p(y) \leftrightarrow x \in y)$.

Set equals Setom Class follows from that $v \in C^\mathcal{M} \cap \mathrm{Setom}^\mathcal{M} \Leftrightarrow \exists x \in N . \exists y \in \mathcal{P}(N) . (t(x) = p(y) = v) \Leftrightarrow \exists x \in N . \exists y \in \mathcal{P}(N) . (t(x) = p(y) = v \wedge x \in S^\mathcal{N}) \Leftrightarrow v \in S^\mathcal{M}$.

It only remains to verify that $\mathcal M$ satisfies $\mathrm{CC}_C$. Let $\phi(z)$ be an $\mathcal{L}_\mathsf{Class}$-formula. Let 
$$A = \{u \in N  \mid \mathcal{M}  \models \phi(t (u))\},$$
and note that $\mathcal{M} \models C(p(A))$.

The following implications complete the proof.
$$
\begin{array}{rl}
& A = \{x \in N  \mid \mathcal{M}  \models \phi(t (x))\} \\
\Rightarrow & \forall x \in N  . \big( x \in A \leftrightarrow \mathcal{M}  \models \phi(t (x)) \big) \\
\Rightarrow & \forall u \in \mathrm{Setom}^\mathcal{M}  . \big( u \in^\mathcal{M} p( A) \leftrightarrow \mathcal{M}  \models \phi(u) \big) \\
\Rightarrow & \forall u \in \mathrm{Setom}^\mathcal{M}  . \big( \mathcal{M}  \models (u \in p(A) \leftrightarrow \phi(u)) \big) \\
\Rightarrow & \mathcal{M}  \models \exists X . \big( C(X) \wedge \forall u \in \mathrm{Setom}. (u \in X \leftrightarrow (\phi(u))) \big) \\
\end{array}
$$

To verify the case $\mathrm{Con}(\mathrm{NF}) \Rightarrow \mathrm{Con}(\mathrm{ML})$, note that if $S^\mathcal{N} = N$, then $\mathrm{Setom}^\mathcal{M} = S^\mathcal{M}$.
\end{proof}

For the predicative version of $\mathrm{ML(U)}$, it suffices to consider the set of definable subsets of a model of $\mathrm{NF(U)}$. Thus, a slightly modified version of the above proof can be carried out for the predicative case in an appropriate set theory of countable sets.

\section{Categorical semantics} \label{CatAxioms}

Categories may be viewed as structures in the basic language of category theory. Traditionally, a theory in the first order language of category theory (or an expansion of that language) is formulated as a definition of a class of models. Such definitions, that can be turned into first order axiomatizations, are called elementary. The definitions of classes of categories made in this section are all easily seen to be elementary. 

Now follows a presentation of the categorical semantics of first order logic in Heyting (intuitionistic logic) and Boolean (classical logic) categories. A full account can be found e.g. in \cite[pp. 807-859]{Joh02}. 

It is assumed that the reader is familiar with basic category theoretic notions: Most importantly, the notions of {\em diagram}, {\em cone}, {\em limit} and their duals (in particular, the special cases of {\em terminal object}, {\em initial object}, {\em product} and {\em pullback}), as well as the notions of {\em functor}, {\em natural transformation} and {\em adjoint} functors.

Since the definition of Heyting categories below uses the notion of adjoint functors between partial orders, let us explicitly define this particular case of adjoint functors: Let $\mathbb{A}$ and $\mathbb{B}$ be partial orders with orderings $\leq_\mathbb{A}$ and $\leq_\mathbb{B}$, respectively. They may be considered as categories with the elements of the partial order as objects, and with a single morphism $x \rightarrow y$ if $x \leq y$, and no morphism from $x$ to $y$ otherwise, for all elements $x, y$ in the partial order. The composition of morphisms is the only one possible. Note that a functor from $\mathbb{A}$ to $\mathbb{B}$, as categories, is essentially the same as an order-preserving function from $\mathbb{A}$ to $\mathbb{B}$, as partial orders. Let $\mathbf{F} : \mathbb{A} \leftarrow \mathbb{B}$ and $\mathbf{G} : \mathbb{A} \rightarrow \mathbb{B}$ be functors. $\mathbf{F}$ is {\em left adjoint} to $\mathbf{G}$, and equivalently $\mathbf{G}$ is {\em right adjoint} to $\mathbf{F}$, written $\mathbf{F} \dashv \mathbf{G}$, if for all objects $X$ in $\mathbb{A}$ and all objects $Y$ in $\mathbb{B}$,
\[
\mathbf{F}Y \leq_\mathbb{A} X \Leftrightarrow Y \leq_\mathbb{B} \mathbf{G}X.
\]

A morphism $f$ is a {\em cover} if whenever $f = m \circ g$ for a mono $m$, then $m$ is an isomorphism. A morphism $f$ has an {\em image} if it factors as $f = m \circ e$, where $m$ is a mono with the universal property that if $f = m' \circ e'$ is some factorization with $m'$ mono, then there is a unique $k$ such that $m = m' \circ k$.

\begin{dfn} \label{HeytingBooleanCategory}
A category is a {\em Heyting category} if it satisfies the following axioms (HC).
\begin{enumerate}
\item[(F1)] It has finite limits.
\item[(F2)] It has images.
\item[(F3)] The pullback of any cover is a cover.
\item[(F4)] Each $\mathrm{Sub}_X$ is a sup-semilattice.
\item[(F5)] For each $f : X \rightarrow Y$, the {\em inverse image functor} $f\hspace{2pt}^* : \mathrm{Sub}_Y \rightarrow \mathrm{Sub}_X$ (defined below) preserves finite suprema and has left and right adjoints: $\exists_f \dashv f\hspace{2pt}^* \dashv \forall_f $.
\end{enumerate}

We call this theory HC. $\mathrm{Sub}_X$ and $f\hspace{2pt}^*$ are explained below. One can prove from these axioms that that each $\mathrm{Sub}_X$ is a Heyting algebra. A {\em Boolean category} is a Heyting category such that each $\mathrm{Sub}_X$ is a Boolean algebra. We call that theory BC.

A {\em Heyting} ({\em Boolean}) {\em functor}, is a functor between Heyting (Boolean) categories that preserves the structure above. $\mathbf{C}$ is a {\em Heyting} ({\em Boolean}) {\em subcategory} of $\mathbf{D}$ if it is a subcategory and the inclusion functor is Heyting (Boolean).
\end{dfn}

Let $\mathbf{C}$ be any Heyting category. It has a terminal object $\mathbf{1}$ and an initial object $\mathbf{0}$, as well as a product $X_1 \times \dots \times X_n$, for any $n \in \mathbb{N}$ (in the case $n = 0$, $X_1 \times \dots \times X_n$ is defined as the the terminal object $\mathbf{1}$). Given an $n \in \mathbb{N}$ and a product $P$ of $n$ objects, the $i$-th projection morphism, for $i = 1, \dots , n$, is denoted $\pi_P^i$ (the subscript $P$ will sometimes be dropped when it is clear from the context). If $f_i : Y \rightarrow X_i$ are morphisms in $\mathbb{C}$, for each $i \in \{1, \dots, n\}$ with $n \in \mathbb{N}$, then $\langle f_1, \dots, f_n\rangle : Y \rightarrow X_1 \times \dots \times X_n$ denotes the unique morphism such that $\pi^i \circ \langle f_1, \dots, f_n\rangle = f_i$, for each $i \in \{1, \dots, n\}$. An important instance of this is that $\mathbf{C}$ has a diagonal mono $\Delta_X : X \rightarrowtail X \times X$, for each $X$, defined by $\Delta_X = \langle \id_X, \id_X \rangle$. If $g_i : Y_i \rightarrow X_i$ are morphisms in $\mathbf{C}$, for each $i \in \{1, \dots, n\}$ with $n \in \mathbb{N}$, then $g_1 \times \dots \times g_n : Y_1 \times \dots \times Y_n \rightarrow X_1 \times \dots \times X_n$ denotes the morphism $\langle g_1 \circ \pi^1, \dots, g_n \circ \pi^n \rangle$.

A {\em subobject} of an object $X$ is an isomorphism class of monos $m : Y \rightarrowtail X$ in the slice category $\mathbf{C}/X$. (Two monos $m : Y \rightarrowtail X$ and $m' : Y\hspace{1pt}' \rightarrowtail X$ are isomorphic in $\mathbf{C}/X$ iff there is an isomorphism $f : Y \rightarrow Y\hspace{1pt}'$ in $\mathbf{C}$, such that $m = m' \circ f$.) It is often convenient to denote such a subobject by $Y$, although it is an abuse of notation; in fact we shall do so immediately. The subobjects of $X$ are endowed with a partial order: If $m : Y \rightarrowtail X$ and $m' : Y\hspace{1pt}' \rightarrowtail X$ represent two subobjects $Y$ and $Y\hspace{1pt}'$ of $X$, then we write $Y \leq_X Y\hspace{1pt}'$ if there is a mono from $m$ to $m'$ in $\mathbf{C}/X$ (i.e. if there is a mono $f : Y \rightarrow Y\hspace{1pt}'$ in $\mathbf{C}$, such that $m = m' \circ f$). 

The axioms (F1)--(F5) ensure that for any object $X$, the partial order of subobjects of $X$, denoted $\mathrm{Sub}(X)$, with its ordering denoted $\leq_X$ and its equality relation denoted $\cong_X$ (or just $=$ when the context is clear), is a Heyting algebra, with constants $\bot_X$, $\top_X$ and operations $\wedge_X$, $\vee_X$, $\rightarrow_X$ (we often suppress the subscript when it is clear from the context). Given a morphism $f : X \rightarrow Y$ in $\mathbf{C}$, the functor $f\hspace{2pt}^* : \mathrm{Sub}(Y) \rightarrow \mathrm{Sub}(X)$ is defined by sending any subobject of $Y$, represented by $m_B : B \rightarrowtail Y$, say, to the subobject of $X$ represented by the pullback of $m_B$ along $f$. Given a subobject $A$ of $Y$, represented by a mono $m_A$ with co-domain $Y$, we may write $A^* : \mathrm{Sub}(Y) \rightarrow \mathrm{Sub}(X)$ as an alternative notation for the functor $m_A^*$.

A {\em structure} (or {\em model}) $\mathcal M$, in the categorical semantics of $\mathbf{C}$, in a sorted signature $\mathcal{S}$, is an assignment of sorts, relation symbols and function symbols of $\mathcal{S}$ to objects, subobjects and morphisms of $\mathbf{C}$, respectively, as now to be explained. 

Sorts: Any sort in $\mathcal{S}$ is assigned to an object of $\mathbf{C}$.

Relation symbols: Any relation symbol $R$ on a sort $S_1 \times \dotsc \times S_n$, where $n \in \mathbb{N}$, is assigned to a subobject $R^\mathcal{M} \leq S^\mathcal{M}_1 \times \dotsc \times S^\mathcal{M}_n$. In particular, the equality symbol $=_S$ on the sort $S \times S$ is assigned to the subobject of $S^\mathcal{M} \times S^\mathcal{M}$ determined by $\Delta_{S^\mathcal{M}} : S^\mathcal{M} \rightarrowtail S^\mathcal{M} \times S^\mathcal{M}$. In the case $n = 0$, $S^\mathcal{M}_1 \times \dotsc \times S^\mathcal{M}_n$ is the terminal object $\mathbf{1}$. Thus, we can handle $0$-ary relation symbols. By the above, such a symbol is assigned to a subobject of $\mathbf{1}$. For example, the unique morphism $\mathbf{1} \rightarrow \mathbf{1}$ and the unique morphism $\mathbf{0} \rightarrow \mathbf{1}$ represent subobjects of $\mathbf{1}$. In the semantics explained below, the former corresponds to truth and the latter corresponds to falsity.

Function symbols: Any function symbol $f : S_1 \times \dotsc \times S_n \rightarrow T$, where $n \in \mathbb{N}$, is assigned to a morphism $f^\mathcal{M} : S^\mathcal{M}_1 \times \dotsc \times S^\mathcal{M}_n \rightarrow T^\mathcal{M}$. Note that in the case $n = 0$, $f$ is assigned to a morphism $1 \rightarrow T$. In this case, we say that $f$ is a {\em constant symbol}.

Let $m, n \in \mathbb{N}$ and let $k \in \{1, \dots, n\}$. The $\mathcal M$-{\em interpretation} $\llbracket \vec{x} : S_1 \times \dotsc \times S_n \mid t \rrbracket^\mathcal{M}$ (which may be abbreviated $\llbracket \vec{x} \mid t \rrbracket$ when the structure and the sorts of the variables are clear) of a term $t$ of sort $T$ in context $\vec x$ of sort $S_1 \times \dotsc \times S_n$ is a morphism $S^\mathcal{M}_1 \times \dotsc \times S^\mathcal{M}_n \rightarrow T^\mathcal{M}$ defined recursively:
\[
\begin{array}{rcl}
\llbracket \vec{x} \mid x_k \rrbracket &=_\mathrm{df}& \pi^k : S^\mathcal{M}_1 \times \dotsc \times S^\mathcal{M}_n \rightarrow S^\mathcal{M}_k \\
\llbracket \vec{x} \mid f\hspace{2pt}(t_1, \dotsc, t_m) \rrbracket &=_\mathrm{df}&
\mathcal M f \circ \langle \llbracket \vec x \mid t_1 \rrbracket, \dotsc, \llbracket \vec{x} \mid t_m \rrbracket \rangle : \\
&& S^\mathcal{M}_1 \times \dotsc \times S^\mathcal{M}_n \rightarrow W^\mathcal{M}, 
\end{array}
\]
where $t_1, \dotsc, t_m$ are terms of sorts $T_1, \dotsc, T_m$, respectively, and $f$ is a function symbol of sort $T_1 \times \dotsc \times T_m \rightarrow W$. 

The $\mathcal M$-{\em interpretation} $\llbracket \vec{x} : S_1 \times \dotsc \times S_n \mid \phi \rrbracket^\mathcal{M}$ (which may be abbreviated $\llbracket \vec{x} \mid \phi \rrbracket$ when the structure and the sorts of the variables are clear) of a formula $\phi$ in context $\vec x$ of sort $S_1 \times \dotsc \times S_n$ is defined recursively:
\[
\begin{array}{rcl}
\llbracket \vec{x} \mid \bot \rrbracket &=_\mathrm{df}& [\bot \rightarrowtail S^\mathcal{M}_1 \times \dotsc \times S^\mathcal{M}_n] \\
\llbracket \vec{x} \mid R(\vec{t}) \rrbracket &=_\mathrm{df}& \llbracket \vec{x} \mid \vec{t} \rrbracket^* (R^\mathcal{M}), \\
&& \text{where $R$ is a relation symbol and $\vec{t}$ are terms.} \\
\llbracket \vec{x} \mid \chi \odot \psi \rrbracket &=_\mathrm{df}& \llbracket \vec{x} \mid \chi \rrbracket \odot \llbracket \vec{x} \mid \psi \rrbracket \text{, where $\odot \in \{\wedge, \vee, \rightarrow\}$.} \\
\llbracket \vec{x}\setminus \{x_k\} \mid \boxminus x_k . \psi \rrbracket &=_\mathrm{df}& \boxminus_{\langle \pi^1, \dotsc, \pi^{k-1}, \pi^{k+1}, \dotsc, \pi^n \rangle} (\llbracket \vec{x} \mid \psi \rrbracket), \\
&& \text{where $\boxminus \in \{\forall, \exists\}$.}
\end{array}
\]
Recall that $\llbracket \vec{x} \mid \vec{t} \rrbracket^*(R^\mathcal{M})$ is obtained by taking the pullback of a representative of $R^\mathcal{M}$ along $\llbracket \vec{x} \mid \vec{t} \rrbracket$. The denotation of $\boxminus_{\langle \pi^1, \dotsc, \pi^{k-1}, \pi^{k+1}, \dotsc, \pi^n \rangle}$ is given in axiom (F5) of Heyting categories above. 

We say that $\phi(\vec x)$ is {\em valid} in $\mathcal{M}$, and write $\mathcal{M} \models \phi$, whenever $\llbracket \vec{x} \mid \phi(\vec{x}) \rrbracket$ equals the maximal subobject $S^\mathcal{M}_1 \times \dotsc \times S^\mathcal{M}_n$ of $\mathrm{Sub}(S^\mathcal{M}_1 \times \dotsc \times S^\mathcal{M}_n)$. In particular, if $\phi$ is a sentence, then $\mathcal{M} \models \phi$ iff $\llbracket \cdot : \cdot \mid \phi \rrbracket = \mathbf{1}$, where the notation ``$\cdot : \cdot$'' stands for the empty sequence of variables in the $0$-ary context. It is of course more convenient to write $\llbracket \cdot : \cdot \mid \phi \rrbracket$ simply as $\llbracket \phi \rrbracket$.

When working with this semantics it is sometimes convenient to use the following well-known rules:
\[
\begin{array}{rcl}
\llbracket \vec{x} \mid \chi \rrbracket \wedge \llbracket \vec{x} \mid \psi \rrbracket &=& \llbracket \vec{x} \mid \chi \rrbracket^*(\llbracket \vec{x} \mid \psi \rrbracket) \\
\llbracket \vec{x} \mid \chi \rightarrow \psi \rrbracket &\Leftrightarrow& \llbracket \vec{x} \mid \chi \rrbracket \leq \llbracket \vec{x} \mid \psi \rrbracket \\
\llbracket \forall x_1 \dots \forall x_n . \psi \rrbracket = \mathbf{1}  &\Leftrightarrow& \llbracket \vec{x} \mid \psi \rrbracket = S^\mathcal{M}_1 \times \dots \times S^\mathcal{M}_1n,
\end{array}
\]
In the last equivalence, it is assumed that $x_1, \dots, x_n$ are the only free variables of $\phi$.
 
When an interpretation $\mathcal{M}$ of $\mathcal{S}$ in a Heyting category $\mathbf{C}$ is given, we will often simply write ``$\mathbf{C} \models \phi$''. Sometimes it is convenient to extend $\mathcal{S}$ with some objects, morphisms and subobjects of $\mathbf{C}$ as new sorts, function symbols and relation symbols, respectively.

\begin{dfn} \label{NaturalSignature}
Let $\mathbf{C}$ be a Heyting category and let $\mathbf{D}$ be a subcategory of $\mathbf{C}$ with finite products. We define the {\em $\mathbf{D}$-signature with respect to $\mathbf{C}$}, denoted $\mathcal{S}^\mathbf{C}_\mathbf{D}$, as the following signature.
\begin{itemize}
\item Sorts: For each object $A$ of $\mathbf{D}$, $A$ is a sort in $\mathcal{S}^\mathbf{C}_\mathbf{D}$.
\item Function symbols: For each morphism $f : A \rightarrow B$ of $\mathbf{D}$, $f : A \rightarrow B$ is a function symbol in $\mathcal{S}^\mathbf{C}_\mathbf{D}$ from the sort $A$ to the sort $B$.
\item Relation symbols: For each $n \in \mathbb{N}$, and for each morphism $m : A \rightarrow B_1 \times \dots \times B_n$ of $\mathbf{D}$, such that $m$ is monic in $\mathbf{C}$, $m$ is an $n$-ary relation symbol in $\mathcal{S}^\mathbf{C}_\mathbf{D}$ on the sort $B_1 \times \dots \times B_n$. (Note that in the case $n = 0$, $B_1 \times \dots \times B_n$ is the terminal object $\mathbf{1}$ of $\mathbf{D}$ and $m$ is a $0$-ary relation symbol.)
\end{itemize}
Given $\mathcal{S}^\mathbf{C}_\mathbf{D}$, the {\em natural $\mathcal{S}^\mathbf{C}_\mathbf{D}$-structure} is defined by assigning each sort $A$ to the object $A$, assigning each function symbol $f$ to the morphism $f$, and assigning each relation symbol $m$ on the sort $B_1 \times \dots \times B_n$ to the subobject of $B_1 \times \dots \times B_n$ in $\mathbf{C}$ determined by $m$. Let $\phi$ be an $\mathcal{S}^\mathbf{C}_\mathbf{D}$-formula. We write $\mathbf{C} \models \phi$ for the statement that $\phi$ is satisfied in the natural $\mathcal{S}^\mathbf{C}_\mathbf{D}$-structure. If no signature has been specified, then $\mathbf{C} \models \phi$ means that $\phi$ is satisfied in the natural $\mathcal{S}^\mathbf{C}_\mathbf{C}$-structure (and it is assumed that $\phi$ is an $\mathcal{S}^\mathbf{C}_\mathbf{C}$-formula).
\end{dfn}

The importance of Heyting categories lies in this well-known result:

\begin{thm}[Completeness for categorical semantics]\label{Completeness}
Intuitionistic and classical first order logic are sound and complete for the categorical semantics of Heyting and Boolean categories, respectively.
\end{thm}

As a first application of the categorical semantics, we shall generalize Proposition \ref{ConSetConClass prop} to the intuitionistic case. This can be done efficiently through the machinery of topos theory.

\begin{dfn}\label{power object}
A {\em topos} is a category with finite limits and power objects. A {\em power object} of an object $A$, is an object $\mathbf{P}A$ along with a mono $m : \hspace{2pt} \in_A \hspace{2pt} \rightarrowtail A \times \mathbf{P}A$ such that for any mono $r : R \rightarrowtail A \times B$, there is a unique morphism $\chi : B \rightarrow \mathbf{P}A$ making this a pullback square:
\[
\begin{tikzcd}[ampersand replacement=\&, column sep=small]
R \ar[rr] \ar[d, rightarrowtail, "r"'] \&\& \in_A \ar[d, rightarrowtail, "m"] \\
A \times B \ar[rr, "{\id \times \chi}"] \&\& A \times \mathbf{P}A 
\end{tikzcd} 
\]
\end{dfn}

The expression ``morphism $\chi : B \rightarrow \mathbf{P}A$ making this a pullback square'' with a pullback-diagram drawn underneath (as above), will be used several times in this text. More formally, it is taken as an abbreviation of ``morphism $\chi : B \rightarrow \mathbf{P}A$ such that $r$ is a pullback of $m$ along $\id \times \chi$'' (where $m$ and $r$ depend as above on the pullback-diagram drawn underneath).

A {\em small} category is a category that can be implemented as a set (i.e. it does not require a proper class). If $\mathbf{C}$ is a small category, then the category $\mathbf{Set}^\mathbf{C}$, of functors from $\mathbf{C}$ to the usual category of sets, with natural transformations as morphisms, is called the {\em category of presheaves of} $\mathbf{C}^\mathrm{op}$.

Here we collect some well-known facts about topoi needed for the proof of Theorem \ref{ToposModelOfML}. Specifically, the last item is Corollary A2.4.3 in \cite{Joh02}.

\begin{prop}\label{topos prop}
Let $\mathbf{C}$ be a small category. Let $\mathbf{Set}$ be the usual category of sets. Let $\mathbf{E}$ be a topos and let $Z$ be an object in $\mathbf{E}$. Let $\mathbf{P}Z$ along with $p_Z : \in_Z \rightarrowtail Z \times \mathbf{P}Z$ be a power object of $Z$ in $\mathbf{E}$.
\begin{enumerate}[{\normalfont (a)}]
\item $\mathbf{Set}^\mathbf{C}$ is a topos.
\item $\mathbf{E}$ is a Heyting category.
\item $\mathbf{E} \models \forall x, y : \mathbf{P}Z . \big( (\forall z : Z . (z \in x \leftrightarrow z \in y)) \rightarrow x = y \big)$
\item For each $\mathcal{S}^\mathbf{E}_\mathbf{E}$-formula $\phi(z, y)$, $\mathbf{E} \models \forall y : Y . \exists x : \mathbf{P}Z . \forall z : Z . (z \in x \leftrightarrow \phi(z, y))$.
\item If $m : A \rightarrow B$ is a mono in $\mathbf{E}$ with a pushout as below,
\[
\begin{tikzcd}[ampersand replacement=\&, column sep=small]
A \ar[rr] \ar[d, rightarrowtail, "m"'] \&\& C \ar[d, "m'"] \\
B \ar[rr] \&\& D 
\end{tikzcd} 
\]
then $m'$ is a mono and the diagram is a pullback.
\end{enumerate}
\end{prop}

An {\em intuitionistic Kripke structure} in a first-order language $\mathcal{L}$ on a partial order $\mathbb{P}$, is an $\mathcal{L}$-structure in the categorical semantics of $\mathbf{Set}^\mathbb{P}$. It is well-known and easily verified that this definition is equivalent to the traditional definition, as given e.g. in \cite{Mos15}.

\begin{thm}\label{ToposModelOfML}
Let $\mathbf{E}$ be a topos. In the categorical semantics of $\mathbf{E}$: If there is a model of $\mathrm{(I)NF(U)}_\mathsf{Set}$, then there is a model of $\mathrm{(I)ML(U)}_\mathsf{Class}$.
\end{thm}
\begin{proof}
This result follows immediately from the proof of Proposition \ref{ConSetConClass prop}, because that proof can literally be carried out in the internal language of any topos. (It is well-known that one can safely reason from the axioms of a weak intuitionistic set theory in this internal language.) However, for the reader's convenience we shall also give the proof in its interpreted form, in the language of category theory.

The intuitionistic and classical cases correspond to the cases that $\mathbf{E}$ is Heyting and Boolean, respectively. The symbol $\in$ is used for the element-relations associated with power objects in $\mathbf{E}$, and use the symbol $\epsin$ for the element-relation symbol in $\mathcal{L}_\mathsf{Set}$ and $\mathcal{L}_\mathsf{Class}$. The object interpreting the domain of $\mathcal{N}$ is denoted $N$. This means that we have a mono $n_S : S^\mathcal{N} \rightarrowtail N$ interpreting the sethood predicate $S$, a morphism $n_P : N \times N \rightarrow N$ interpreting the pairing function $\langle -, - \rangle$, and a mono $n_{\epsin} : \epsin^\mathcal{N} \rightarrow N \times N$ interpreting the element-relation $\epsin$. 

Sethood: $\mathcal{N} \models z \epsin x \rightarrow S(x)$ means that 
$$\llbracket x, y : N \mid x \epsin y \rrbracket \leq \llbracket x, y : N \mid S(y) \rrbracket = N \times S^\mathcal{N},$$ 
so there is a mono $n'_{\epsin} : \epsin^\mathcal{N} \rightarrowtail N \times S^\mathcal{N}$, such that $(\id_N \times n_S) \circ n'_{\epsin} = n_{\epsin}$.

Ordered Pair: $\mathcal{N} \models \langle x, y \rangle = \langle x\hspace{1pt}', y\hspace{1pt}' \rangle \rightarrow ( x = x\hspace{1pt}' \wedge y = y\hspace{1pt}' )$ means that $n_P : N \times N \rightarrowtail N$ is monic.

$\mathrm{Ext}_S$: $\mathcal{N} \models (S(x) \wedge S(y) \wedge \forall z . (z \epsin x \leftrightarrow z \epsin y)) \rightarrow x = y$ implies that for any pullback-square of the form below, $\chi$ is the unique morphism making this a pullback-square:
\[
\begin{tikzcd}[ampersand replacement=\&, column sep=small]
R \ar[rr] \ar[d, rightarrowtail, "r"'] \&\& {\epsin^\mathcal{N}} \ar[d, rightarrowtail, "n'_{\epsin}"] \\
N \times B \ar[rr, "{\id \times \chi}"] \&\& N \times S^\mathcal{N}
\end{tikzcd} 
\tag{A}
\]
To see this, we shall work with the natural $\mathcal{S}^{\mathbf{E}}_{\mathbf{E}}$-structure, which expands $\mathcal{N}$. Let $b, b\hspace{1pt}' : B \rightarrow S^\mathcal{N}$, such that $r$ is a pullback of $n'_{\epsin}$, both along $\id \times b$ and along $\id \times b\hspace{1pt}'$. By the categorical semantics, $r$ then represents both $\llbracket z : N, v : B \mid z \epsin b(v) \rrbracket$ and $\llbracket z : N, v : B \mid z \epsin b\hspace{1pt}'(v) \rrbracket$. So 
$$\mathbf{E} \models \forall v : B . \forall z : N . (z \epsin b(v) \leftrightarrow z \epsin b\hspace{1pt}'(v)),$$ whence by $\mathbf{E} \models \mathrm{Ext}_S$, we have $\mathbf{E} \models \forall v : B . b(v) = b\hspace{1pt}'(v)$. It follows that $b = b\hspace{1pt}'$.

$\mathrm{SC}_S$: For all stratified $\phi$, $\mathcal{N} \models \setmany z . \phi(z, y)$. Although this remark is not needed for the proof, it may help to clarify: $\mathcal{N} \models \mathrm{SC}_S$ implies that for any stratified $\mathcal{L}_\mathsf{Set}$-formula $\phi(z, y)$, there is a morphism $\chi : N \rightarrow S^\mathcal{N}$ making this a pullback-square:
\[
\begin{tikzcd}[ampersand replacement=\&, column sep=small]
{\llbracket z, y : N \mid \phi(z, y) \rrbracket} \ar[rr] \ar[d, rightarrowtail, "r"'] \&\& {\epsin^\mathcal{N}} \ar[d, rightarrowtail, "n'_{\epsin}"] \\
N \times N \ar[rr, "{\id \times \chi}"] \&\& N \times S^\mathcal{N}
\end{tikzcd}
\]

By the pullback-property of the power object $\mathbf{P}N$ of $N$, there is a unique $\chi_S$ making this a pullback-square:
\[
\begin{tikzcd}[ampersand replacement=\&, column sep=small]
{\epsin^\mathcal{N}} \ar[rr] \ar[d, rightarrowtail, "n'_{\epsin}"'] \&\& {\in_N} \ar[d, rightarrowtail, "n_{\in_N}"] \\
N \times S^\mathcal{N} \ar[rr, "{\id \times \chi_S}"] \&\& N \times \mathbf{P}N
\end{tikzcd} 
\tag{B}
\]
By combining (A) with (B), we find that $\chi_S$ is monic: Let $b, b\hspace{1pt}' : B \rightarrow S^\mathcal{N}$, such that $\chi_S \circ b = \chi_S \circ b\hspace{1pt}'$. Let $r : R \rightarrowtail N \times B$ and $r' : R' \rightarrowtail N \times B$ be the pullbacks of $n'_{\epsin}$ along $\id_N \times b$ and $\id_N \times b\hspace{1pt}'$, respectively. Consider these pullbacks as instances of (A) above. ``Gluing\hspace{1pt}'' each of these pullback-diagrams with (B) along the common morphism $n'_{\epsin}$, yields two new pullback-diagrams with the bottom morphisms $\mathrm{id}_N \times (\chi_S \circ b)$ and $\mathrm{id}_N \times (\chi_S \circ b\hspace{1pt}')$, respectively. (It is a basic and well-known property of pullbacks that such a ``gluing\hspace{1pt}'' of two pullback yields another pullback.) We know that these bottom morphisms are equal. Thus, by uniqueness of pullbacks up to isomorphism, we may assume that $r = r'$ and $R = R'$. Now it follows from the uniqueness of $\chi$ in (A) that $b = b\hspace{1pt}'$.

We proceed to construct an $\mathcal{L}_\mathsf{Class}$-structure $\mathcal{M}$ in $\mathbf{E}$, such that $\mathcal{M} \models \mathrm{(I)MLU}$. The domain $M$ of the structure is constructed as this pushout:
\[
\begin{tikzcd}[ampersand replacement=\&, column sep=small]
S^\mathcal{N} \ar[rr, rightarrowtail, "\chi_S"] \ar[d, rightarrowtail, "n_S"'] \&\& {\mathbf{P}N} \ar[d, rightarrowtail, "m_C"] \\
N \ar[rr, rightarrowtail, "m_{\mathrm{Setom}}"] \&\& M
\end{tikzcd} 
\tag{C}
\]
By Proposition \ref{topos prop}, $m_{\mathrm{Setom}}$ and $m_C$ are monic.

We interpret the predicate $\mathrm{Setom}$ by the mono $m_{\mathrm{Setom}} : N \rightarrowtail M$ and the classhood predicate $C$ by the mono $m_C : \mathbf{P}N \rightarrowtail M$. Naturally, we interpret $S$ by the mono $m_S =_\mathrm{df} n_\mathrm{Setom} \circ n_S : S^\mathcal{N} \rightarrowtail M$ and the partial ordered pair function by the mono $m_P =_\mathrm{df} m_{\mathrm{Setom}} \circ n_P : N \times N \rightarrow M$. (Formally, ordered pair is treated as a ternary relation symbol in $\mathcal{L}_\mathsf{Class}$, which is interpreted by $\llbracket x, y, z : M \mid \exists x\hspace{1pt}', y\hspace{1pt}', z\hspace{1pt}' : N . (m_{\mathrm{Setom}}(x\hspace{1pt}') = x \wedge m_{\mathrm{Setom}}(y\hspace{1pt}') = y \wedge m_{\mathrm{Setom}}(z\hspace{1pt}') = z \wedge m_P(x\hspace{1pt}', y\hspace{1pt}') = z\hspace{1pt}' \rrbracket$.)

We interpret the element-relation $\epsin$ by the mono
\begin{align*}
m_{\epsin} &=_\mathrm{df} (m_{\mathrm{Setom}} \times \id_M) \circ (\id_N \times m_C) \circ n_{\in_N} \\
&= (m_{\mathrm{Setom}} \times m_C) \circ n_{\in_N},
\end{align*} 
from  $\in_N$ to $M \times M$. Moreover, let $m'_{\epsin} = (\id_N \times m_C) \circ n_{\in_N}$, so that $m_{\epsin} = (m_{\mathrm{Setom}} \times \id_M) \circ m'_{\epsin}$. Now consider this diagram, obtained by gluing (B) on top of the diagram ``$N \times$ (C)'':
\[
\begin{tikzcd}[ampersand replacement=\&, column sep=small]
{\epsin^\mathcal{N}} \ar[rr] \ar[d, rightarrowtail, "n'_{\epsin}"'] \&\& {\in_N} \ar[d, rightarrowtail, "n_{\in_N}"] \\
N \times S^\mathcal{N} \ar[rr, rightarrowtail, "\id \times \chi_S"] \ar[d, rightarrowtail, "\id \times n_S"'] \&\& {N \times \mathbf{P}N} \ar[d, rightarrowtail, "\id \times m_C"] \\
N \times N \ar[rr, rightarrowtail, "\id \times m_{\mathrm{Setom}}"] \&\& N \times M
\end{tikzcd} 
\tag{D}
\]

The lower square is a pushout because (C) is, so by Proposition \ref{topos prop} it is a pullback. Since (B) is also a pullback, we have by a basic well-known result that (D) is also a pullback. It follows that $\llbracket z, x : N \mid z \epsin x \rrbracket ^\mathcal{N} \cong_{N \times N} \llbracket z, x : N \mid m_{\mathrm{Setom}}(z) \epsin m_{\mathrm{Setom}}(x) \rrbracket ^\mathcal{M}$. In other words, the interpretations of $\epsin$ in $\mathcal{N}$ and $\mathcal{M}$ agree on $N$ as a subobject of $M$ represented by $m_\mathrm{Setom}$. We can now easily verify the axioms of $\mathrm{(I)MLU}_\mathsf{Class}$.

Classhood: $\mathcal{M} \models z \epsin x \rightarrow C(x)$ follows from that $m_{\epsin} = (m_{\mathrm{Setom}} \times \id_M) \circ (\id_N \times m_C) \circ n_{\in_N}$.

Setomhood: $\mathcal{M} \models z \epsin x \rightarrow \mathrm{Setom}(x)$ also follows from that $m_{\epsin} = (m_{\mathrm{Setom}} \times \id_M) \circ (\id_N \times m_C) \circ n_{\in_N}$.

$\mathrm{Ext}_C$: $\mathcal{M} \models (C(x) \wedge C(y) \wedge \forall z . (z \epsin x \leftrightarrow z \epsin y)) \rightarrow x = y$ follows from that $\mathbf{E} \models \forall x, y : \mathbf{P}N . \big( (\forall z : N . (z \in x \leftrightarrow z \in y)) \rightarrow x = y \big)$ (see Proposition \ref{topos prop}), that $m_\mathrm{Setom}$ is monic, and that $m_{\epsin} = (m_{\mathrm{Setom}} \times \id_M) \circ (\id_N \times m_C) \circ n_{\in_N}$.

$\mathrm{CC}_C:$ For all $\mathcal{L}_\mathsf{Class}$-formulae $\phi$, $\mathcal{M} \models \exists x . \big( C(x) \wedge \forall z \in \mathrm{Setom} . (z \epsin x \leftrightarrow \phi(z))\big)$, follows from that $\mathbf{E} \models \exists x : \mathbf{P}N . \forall z : N . (z \in x \leftrightarrow \phi(z))$ and that $m_{\epsin} = (m_{\mathrm{Setom}} \times \id_M) \circ (\id_N \times m_C) \circ n_{\in_N}$.

$\mathrm{SC}_S:$ For all stratified $\mathcal{L}_\mathsf{Class}$-formulae $\phi$ with only $z, \vec{y}$ free, $\mathcal{M} \models \forall \vec{y} \in \mathrm{Setom} . \setmany z . \phi(z, \vec{y})$, follows from that $\mathcal{N} \models \mathrm{SC}_S$, and that 
$$\llbracket x, y : N \mid x \epsin y \rrbracket ^\mathcal{N} \cong_{N \times N} \llbracket x, y : N \mid m_{\mathrm{Setom}}(x) \epsin m_{\mathrm{Setom}}(y) \rrbracket ^\mathcal{M}.$$

Ordered Pair: $\mathcal{M} \models \forall x, y, x\hspace{1pt}', y\hspace{1pt}' \in \mathrm{Setom} . ( \langle x, y \rangle = \langle x\hspace{1pt}', y\hspace{1pt}' \rangle \leftrightarrow ( x = x\hspace{1pt}' \wedge y = y\hspace{1pt}' ) )$, follows from that $\mathcal{N} \models \textnormal{Ordered Pair}$, that $m_\mathrm{Setom}$ is monic, and that $m_P = m_\mathrm{Setom} \circ n_P$.

Set equals Setom Class: $\mathcal{M} \models S(x) \leftrightarrow (\mathrm{Setom}(x) \wedge C(x))$, follows from that $m_S = m_\mathrm{Setom} \circ n_S$ and that $n_S$ is a pullback of $m_C$ along $m_\mathrm{Setom}$, as seen in diagram (C).

This concludes the verification of $\mathcal{M} \models \mathrm{(I)MLU}_\mathsf{Class}$. For the case without atoms, note that if $ \mathcal{N} \models \forall x . S(x)$, then $n_S$ is an iso, so since $m_S = m_\mathrm{Setom} \circ n_S$, we have $\mathcal{M} \models \forall x . (S(x) \leftrightarrow \mathrm{Setom}(x))$.
\end{proof}

\begin{cor}\label{ConSetConClass}
$\mathrm{(I)NF(U)}_\mathsf{Set}$ is equiconsistent to $\mathrm{(I)ML(U)}_\mathsf{Class}$.
\end{cor}
\begin{proof}
The $\Leftarrow$ direction was established directly after Axioms \ref{AxIMLU}. For the $\Rightarrow$ direction: By the completeness theorem for intuitionistic predicate logic and Kripke models, there is a Kripke model of $\mathrm{(I)NF(U)}_\mathsf{Set}$, i.e. there is a partial order $\mathbb{P}$ and an $\mathcal{L}_\mathsf{Set}$-structure $\mathcal{N}$ in $\mathbf{Set}^\mathbb{P}$, such that $\mathcal{N} \models \mathrm{(I)NF(U)}_\mathsf{Set}$. By Proposition \ref{topos prop}, $\mathbf{Set}^\mathbb{P}$ is a topos, so it follows from Theorem \ref{ToposModelOfML} that there is a Kripke model of $\mathrm{(I)ML(U)}_\mathsf{Class}$. The classical cases are obtained by setting $\mathbb{P}$ to a singleton.
\end{proof}

\begin{rem}\label{remML}
An equiconsistency statement is trivial unless the consistency strength of the theories considered is at least that of the meta-theory. It is folklore that the consistency strength of $\mathrm{NFU}_\mathsf{Set}$ is at least that of a weak set theory called Mac Lane set theory (by \cite{Jen69} it is at most that), and that the category of presheaves is a topos with Mac Lane set theory as meta-theory, so for the classical case the statement is non-trivial. Moreover, if one unpacks the above equiconsistency proof, one finds that the full Powerset axiom is not needed. It suffices that powersets of countable sets exist, to construct the needed Kripke structure. The strengths of $\mathrm{INF(U)}_\mathsf{Set}$ have not been studied much, so the non-triviality of the above equiconsistency statement needs to be taken as conditional in that case. Regardless of these matters, the proof of Theorem \ref{ToposModelOfML} is constructive and yields information on the close relationship between $\mathrm{(I)NF(U)_\mathsf{Set}}$ and $\mathrm{(I)ML(U)}_\mathsf{Class}$ (also in the intuitionistic case).
\end{rem}

\section{Stratified categories of classes}\label{Formulation of ML_CAT}

We now proceed to introduce a new categorical theory, intended to characterize the categorical content of predicative $\mathrm{(I)ML(U)}_\mathrm{Class}$. For comparison, let us first recall the definition of topos.

We need a couple of more definitions: A functor $\mathbf{F} : \mathbf{C} \rightarrow \mathbf{D}$ is {\em conservative} if for any morphism $f$ in $\mathbf{C}$, if $\mathbf{F}(f\hspace{2pt})$ is an isomorphism then $f$ is an isomorphism. A subcategory is {\em conservative} if its inclusion functor is conservative. A {\em universal object} in a category $\mathbf{C}$ is an object $X$, such that for every object $Y$ there is a mono $f : Y \rightarrowtail X$. The theory $\mathrm{IMLU}_\mathsf{Cat}$ is axiomatized as follows. 

\begin{dfn}[$\mathrm{IMLU}_\mathsf{Cat}$] \label{IMLU_CAT}
A {\em stratified category of classes} (or an $\mathrm{IMLU}$-category) is a pair of Heyting categories $( \mathbf{M}, \mathbf{N})$, such that
\begin{itemize}
\item $\mathbf{N}$ is a conservative Heyting subcategory of $\mathbf{M}$, 
\item there is an object $U$ in $\mathbf{N}$ which is universal in $\mathbf{N}$,
\item there is an endofunctor $\mathbf{T}$ on $\mathbf{M}$, restricting to an endofunctor of $\mathbf{N}$ (also denoted $\mathbf{T}$), along with a natural isomorphism $\iota : \id_\mathbf{M} \xrightarrow{\sim} \mathbf{T}$ on $\mathbf{M}$, 
\item there is an endofunctor $\mathbf{P}$ on $\mathbf{N}$, such that for each object $A$ in $\mathbf{N}$, $\mathbf{T}A$ has a power object $\mathbf{P}A$, $m_{\subseteq^{\mathbf{T}}_A} : \hspace{1pt} \subseteq^{\mathbf{T}}_A \hspace{1pt} \rightarrowtail \mathbf{T} A \times \mathbf{P} A$ in $\mathbf{N}$,
\item  there is a natural isomorphism $\mu : \mathbf{P} \circ \mathbf{T} \xrightarrow{\sim} \mathbf{T} \circ \mathbf{P}$ on $\mathbf{N}$.
\end{itemize}

If ``Heyting\hspace{1pt}'' is replaced with ``Boolean'' throughout the definition, then we obtain the theory $\mathrm{MLU}_\mathsf{Cat}$. If $U \cong \mathbf{P}U$ is added to $\mathrm{(I)MLU}_\mathsf{Cat}$, then we obtain the theories $\mathrm{(I)ML}_\mathsf{Cat}$, respectively.

For convenience we spell out the power object condition in detail: $m_{\subseteq^{\mathbf{T}}_A} : \hspace{1pt} \subseteq^{\mathbf{T}}_A \hspace{1pt} \rightarrowtail \mathbf{T} A \times \mathbf{P} A$ is a mono in $\mathbf{N}$, such that for any mono $r : R \rightarrowtail \mathbf{T} A \times B$ in $\mathbf{N}$, there is a unique $\chi : B \rightarrow \mathbf{P} A$ in $\mathbf{N}$, making this a pullback square in $\mathbf{N}$:
\begin{equation} \label{PT}
\begin{tikzcd}[ampersand replacement=\&, column sep=small]
R \ar[rr] \ar[d, rightarrowtail, "r"'] \&\& \subseteq^{\mathbf{T}}_A \ar[d, rightarrowtail, "{m_{\subseteq^{\mathbf{T}}_A}}"] \\
\mathbf{T} A \times B \ar[rr, "{\id \times \chi}"] \&\& \mathbf{T} A \times \mathbf{P} A 
\end{tikzcd} 
\tag{PT}
\end{equation}
\end{dfn}

In order to carry over some intuitions from a stratified set theory such as $\mathrm{NFU}$, $\mathbf{T}A$ may be thought of as $\{\{x\} \mid x \in A\}$ and $\mathbf{P}A$ may be thought of as $\{ X \mid X \subseteq A\}$. Now $\subseteq^{\mathbf{T}}_A$ corresponds to the subset relation on $\mathbf{T}A \times \mathbf{P}A$. Note that on this picture, $\subseteq^{\mathbf{T}}$ is very similar to the $\in$-relation. Thus (\ref{PT}) is intended to be the appropriate variant for stratified set theory of the power object axiom of topos theory. These intuitions are made precise in the proof of Theorem \ref{ConClassConCat}, where we interpret $\mathrm{(I)ML(U)}_\mathsf{Cat}$ in $\mathrm{(I)ML(U)}_\mathsf{Class}$.

It is easily seen that this axiomatization is elementary, i.e. it corresponds to a theory in a first order language $\mathcal{L}_\mathsf{Cat}$. Its precise specification involves quite some detail. Suffice to say that the language of category theory is augmented with relation symbols $\mathbf{M}_\mathrm{Ob}$, $\mathbf{M}_\mathrm{Mor}$, $\mathbf{N}_\mathrm{Ob}$, and $\mathbf{N}_\mathrm{Mor}$; a constant symbol $U$; and function symbols $\mathbf{T}_\mathrm{Ob}$, $\mathbf{T}_\mathrm{Mor}$, $\iota$, $\mu$, $\mathbf{P}_\mathrm{Ob}$ and $\mathbf{P}_\mathrm{Mor}$ (using the same names for the symbols and their interpretations, and where the subscripts $\mathrm{Ob}$ and $\mathrm{Mor}$ indicate the component of the functor acting on objects and morphisms, respectively).

Note that the definition can easily be generalized, so that we merely require that $\mathbf{N}$ is a Heyting category that is mapped into $\mathbf{M}$ by a faithful conservative Heyting functor $\mathbf{F} : \mathbf{N} \rightarrow \mathbf{M}$. This would not hinder any of the results below. We choose the more specific definition in terms of a subcategory because it simplifies the statements of the results.

We need a relativized notion of power object.

\begin{dfn}
Let $\mathbf{C}$ be a category, and let $\mathbf{D}$ be a subcategory of $\mathbf{C}$. A {\em power object in} $\mathbf{C}$ {\em with respect to} $\mathbf{D}$, of an object $A$ in $\mathbf{D}$, is defined as in Definition \ref{power object}, except that $r$ is assumed to be in $\mathbf{D}$ and $m, \chi$ are required to be in $\mathbf{D}$. More precisely, it is an object $\mathbf{P}A$ along with a morphism $m : \hspace{2pt} \in \hspace{2pt} \rightarrowtail A \times \mathbf{P}A$ in $\mathbf{D}$ which is monic in $\mathbf{C}$, such that for any $r : R \rightarrowtail A \times B$ in $\mathbf{D}$ which is monic in $\mathbf{C}$, there is a morphism $\chi : B \rightarrow \mathbf{P}A$ in $\mathbf{D}$, which is the unique morphism in $\mathbf{C}$ making this a pullback square in $\mathbf{C}$:
\[
\begin{tikzcd}[ampersand replacement=\&, column sep=small]
R \ar[rr] \ar[d, rightarrowtail, "r"'] \&\& \in \ar[d, rightarrowtail, "m"] \\
A \times B \ar[rr, "{\id \times \chi}"] \&\& A \times \mathbf{P}A 
\end{tikzcd} 
\]
\end{dfn}

We shall now collect a few useful properties of $\mathrm{(I)ML(U)}$-categories. First a definition: A functor $\mathbf{F} : \mathbf{B} \rightarrow \mathbf{C}$ {\em reflects finite limits} if for any finite diagram $\mathbf{D} : \mathbf{I} \rightarrow \mathbf{B}$ and for any cone $\Lambda$ of $\mathbf{D}$ in $\mathbf{B}$, if $\mathbf{F}\Lambda$ is a limit in $\mathbf{C}$ of $\mathbf{F} \circ \mathbf{D} : \mathbf{I} \rightarrow \mathbf{C}$, then $\Lambda$ is a limit of $\mathbf{D} : \mathbf{I} \rightarrow \mathbf{B}$ in $\mathbf{B}$.

\begin{prop} \label{strenghten def}
Let $( \mathbf{M}, \mathbf{N})$ along with $U$, $\mathbf{T}$, $\iota$, $\mathbf{P}$ and $\mu$ be an $\mathrm{IMLU}$-category.
\begin{enumerate}[{\normalfont (a)}]
\item \label{strengthen mono} For any morphism $f : A \rightarrow B$ in $\mathbf{N}$, $f$ is monic in $\mathbf{N}$ iff $f$ is monic in $\mathbf{M}$.
\item \label{strengthen reflect limits} The inclusion functor of $\mathbf{N}$ (as a subcategory) into $\mathbf{M}$ reflects finite limits.
\item \label{strengthen power} For any $A$ in $\mathbf{N}$, $\mathbf{P}A$ along with $m_{\subseteq^{\mathbf{T}}_A}$, as in {\em (\ref{PT})} above, is a power object of $\mathbf{T}A$ in $\mathbf{M}$ with respect to $\mathbf{N}$.
\item \label{strengthen T heyting} $\mathbf{T} : \mathbf{M} \rightarrow \mathbf{M}$ is a Heyting endofunctor. If $( \mathbf{M}, \mathbf{N})$ is an $\mathrm{MLU}$-category, then $\mathbf{T} : \mathbf{M} \rightarrow \mathbf{M}$ is a Boolean endofunctor.
\item \label{strengthen T limits} $\mathbf{T} : \mathbf{N} \rightarrow \mathbf{N}$ preserves finite limits.
\end{enumerate}
\end{prop}
\begin{proof}
(\ref{strengthen mono}) ($\Leftarrow$) follows immediately from that $\mathbf{N}$ is a subcategory of $\mathbf{M}$. ($\Rightarrow$) follows from that $\mathbf{N}$ is a Heyting subcategory of $\mathbf{M}$, and that Heyting functors preserve pullbacks, because in general, a morphism $m : A \rightarrow B$ is monic iff $\id_A : A \rightarrow A$ and $\id_A : A \rightarrow A$ form a pullback of $m$ and $m$ (as is well known and easy to check).

(\ref{strengthen reflect limits}) Let $L$, along with some morphisms in $\mathbf{N}$, be a cone in $\mathbf{N}$ of a finite diagram $\mathbf{D} : \mathbf{I} \rightarrow \mathbf{N}$, such that this cone is a limit of $\mathbf{D} : \mathbf{I} \rightarrow \mathbf{N}$ in $\mathbf{M}$. Let $K$ be a limit in $\mathbf{N}$ of $\mathbf{D} : \mathbf{I} \rightarrow \mathbf{N}$. Since $\mathbf{N}$ is a Heyting subcategory of $\mathbf{M}$, $K$ is also such a limit in $\mathbf{M}$. Let $f : L \rightarrow K$ be the universal morphism in $\mathbf{N}$ obtained from the limit property of $K$ in $\mathbf{N}$. By the limit properties of $K$ and $L$ in $\mathbf{M}$, $f$ is an isomorphism in $\mathbf{M}$. Since $\mathbf{N}$ is a conservative subcategory of $\mathbf{M}$, $f$ is also an isomorphism in $\mathbf{N}$, whence $L$ is a limit in $\mathbf{N}$ of $\mathbf{D} : \mathbf{I} \rightarrow \mathbf{N}$, as desired.

(\ref{strengthen power}) Since $\mathbf{N}$ is a Heyting subcategory, any pullback in $\mathbf{N}$ is also a pullback in $\mathbf{M}$. Together with (\ref{strengthen mono}), this easily yields that every power object in $\mathbf{N}$ is a power object in $\mathbf{M}$ with respect to $\mathbf{N}$.

(\ref{strengthen T heyting}) Since $\mathbf{T} : \mathbf{M} \rightarrow \mathbf{M}$ is naturally isomorphic to the identity functor, which is trivially a Heyting (Boolean) functor, $\mathbf{T}$ is also a Heyting (Boolean) endofunctor of $\mathbf{M}$.

(\ref{strengthen T limits}) Let $L$ be a limit in $\mathbf{N}$ of a finite diagram $\mathbf{D} : \mathbf{I} \rightarrow \mathbf{N}$. By (\ref{strengthen T heyting}), $\mathbf{T} : \mathbf{M} \rightarrow \mathbf{M}$ preserves limits, so $\mathbf{T}L$ is a limit in $\mathbf{M}$ of $\mathbf{T} \circ \mathbf{D} : \mathbf{I} \rightarrow \mathbf{N}$. By (\ref{strengthen reflect limits}), $\mathbf{T}L$ is also a limit in $\mathbf{N}$ of $\mathbf{T} \circ \mathbf{D} : \mathbf{I} \rightarrow \mathbf{N}$.
\end{proof}

We now proceed to show $\mathrm{Con}(\mathrm{(I)NF(U)}_\mathsf{Class}) \Rightarrow \mathrm{Con}(\mathrm{(I)ML(U)}_\mathsf{Cat})$. This is the easy and perhaps less interesting part of the equiconsistency proof, but it has the beneficial spin-off of showing how the axioms of $\mathrm{(I)ML(U)}_\mathsf{Cat}$ correspond to set theoretic intuitions. Given Corollary \ref{ConSetConClass}, it suffices to find an interpretation of $\mathrm{(I)ML(U)}_\mathsf{Cat}$ in $\mathrm{(I)ML(U)}_\mathsf{Class}$, as is done in the proof below. This proof actually shows that $\mathrm{(I)ML(U)}_\mathsf{Cat}$ can be interpreted in {\em predicative} $\mathrm{(I)ML(U)}_\mathsf{Class}$; the formulae used in the class-abstracts of the proof only need quantifiers bounded to the extension of $\mathrm{Setom}$.

\begin{thm}\label{ConClassConCat}
$\mathrm{(I)ML(U)}_\mathsf{Cat}$ is interpretable in $\mathrm{(I)ML(U)}_\mathsf{Class}$.
\end{thm}
\begin{proof}
We go through the case of $\mathrm{IMLU}$ in detail, and then explain the modifications required for the other cases. Throughout the interpretation, we work in $\mathrm{IMLU}_\mathsf{Class}$, introducing class and set abstracts $\{x \mid \phi(x, p)\}$, whose existence are justified by the axioms $\mathrm{CC}_C$ and $\mathrm{SC}_S$, respectively. Such class and set abstracts satisfy $\forall x . ( x \in \{x\hspace{1pt}' \mid \phi(x\hspace{1pt}', p)\} \leftrightarrow (\phi(x, p) \wedge x  \in \mathrm{Setom}) )$. Whenever $\phi(x, p)$ is stratified and we have $\mathrm{Setom}(p)$, then the corresponding set abstract exists (and is also a class). Because of the stratification constraint on ordered pairs, when showing that a function $(x \mapsto y)$ defined by $\phi(x, y, p)$ is coded as a set, we have to verify that $\phi(x, y, p)$ can be stratified with the same type assigned to $x$ and $y$. There are no constraints on $\phi$, for a class abstract to exist. Throughout the proof, these $\phi$ are written out explicitly, but for the most part the stratification verifications are simple and left to the reader. 

The interpretation proceeds as follows.
	\begin{enumerate}
	\item Interpret $\mathbf{M}_\mathrm{Ob}(x)$ as $C(x)$, i.e. ``$x$ is a class''.
	\item Interpret $\mathbf{M}_\mathrm{Mor}(m)$ as ``$m$ is a disjoint union of three classes $A$, $B$ and $f$, such that $f$ is a set of pairs coding a function with domain $A$ and co-domain $B$''.\footnote{A disjoint union of three classes may be implemented as a class using the formula $\langle i, x \rangle \in m \leftrightarrow \big( (i = 1 \wedge x \in A) \vee (i = 2 \wedge x \in B) \vee (i = 3 \wedge x \in f\hspace{2pt})\big)$. In order to be able to interpret the domain and co-domain function symbols, we need to include information about the domain class and co-domain class in the interpretation of the morphisms. Otherwise, the same functional class will often interpret many morphisms with different co-domains.} For convenience, we extend the functional notation to $m$ in this setting, i.e. $m(x) =_\mathrm{df} f\hspace{2pt}(x)$, for all $x \in A$, and we also say that $m$ codes this function/morphism from $A$ to $B$.
	\item Interpret the remaining symbols of the language of category theory in the obvious way. Most importantly, composition of morphisms is interpreted by composition of functions. The resulting interpretations of the axioms of category theory are now easily verified for $\mathbf{M}$.
	\item Interpret $\mathbf{N}_\mathrm{Ob}(x)$ as ``$x$ is a set''; and interpret $\mathbf{N}_\mathrm{Mor}(m)$ as ``[insert the interpretation of $\mathbf{M}_\mathrm{Mor}(m)$] and $m$ is a set''. The axioms of category theory are now easily verified for $\mathbf{N}$.
	\item We need to show that the interpretation of the axioms of Heyting categories hold for $\mathbf{M}$ and $\mathbf{N}$. It is well-known that these axioms hold for the categories of classes and sets in conventional class and set theory, see for example \cite{Gol06}. Here we use the same class and set constructions, we just need to check that the axioms $\mathrm{CC}_C$ and $\mathrm{SC}_S$ of $\mathrm{IMLU}_\mathsf{Cat}$ are strong enough to yield the needed sets. $\mathrm{Ext}_C$ ensures the uniqueness conditions in the axioms. 
	
	Existence conditions are supported by class/set abstracts $\{x \mid \phi(x)\}$, where the formula $\phi$ is stratified. We write out each such $\phi$ explicitly and let the reader do the simple verification that $\phi$ is stratified. Thus, in the case of $\mathbf{N}$ we can rely on $\mathrm{SC}_S$, and in the case of $\mathbf{M}$ we can rely on CC$_C$. The only difference is that in the latter case the formula $\phi$ in the class abstract may have parameters which are proper classes. So we can do the verifications for $\mathbf{M}$ and $\mathbf{N}$ simultaneously. 
	
	Let $m : A \rightarrow B$ and $n : C \rightarrow B$ be morphisms in $\mathbf{M}$ or $\mathbf{N}$. Note that for $\mathbf{M}$ and $\mathbf{N}$, subobjects are represented by subclasses and subsets, respectively. Moreover, in both $\mathbf{M}$ and $\mathbf{N}$, any morphism is monic iff injective, and is a cover iff surjective. 
		\begin{enumerate}
		\item[(F1)] Finite limits: It is well-known that the existence of all finite limits follows from the existence of a terminal object and the existence of all pullbacks. $\{\varnothing\}$ is a terminal object. $D =_\mathrm{df} \{ \langle x, z  \rangle \in A \times C \mid m(x) = n(z)\}$, along with the restricted projection morphisms $\pi^1\restriction_D : D \rightarrow A$ and $\pi^2\restriction_D : D \rightarrow C$, is a pullback of the morphisms $m : A \rightarrow B$ and $n : C \rightarrow B$.
		\item[(F2)] Images: The class or set $\{m(x) \mid x \in A\} \subseteq B$, along with its inclusion function into $B$ is the image of $m : A \rightarrow B$.
		\item[(F3)] The pullback of any cover is a cover: Consider the pullback of $m$ and $n$ considered above, and suppose that $m$ is surjective. Then, for any $c \in C$, there is $a \in A$ such that $m(a) = n(c)$, whence $\langle a, c \rangle \in D$. So the projection $D \rightarrow C$ is surjective, as required.
		\item[(F4)] Each $\mathrm{Sub}_X$ is a sup-semilattice under $\subseteq$: Since subobjects are represented by subclasses/subsets, each $\mathrm{Sub}_X$ is the partial order of subclasses/subsets of $X$. Binary union, given by the set abstract $\{z \mid z \in A \vee z \in B\}$, yields the binary suprema required for $\mathrm{Sub}_X$ to be a sup-semilattice. (Note that $\mathrm{Sub}_X$ does not need to be implemented as a set or a class.)
		\item[(F5)] For each morphism $f : X \rightarrow Y$, the functor $f\hspace{2pt}^* : \mathrm{Sub}_Y \rightarrow \mathrm{Sub}_X$ preserves finite suprema and has left and right adjoints, $\exists_f \dashv f\hspace{2pt}^* \dashv \forall_f\hspace{2pt} $: 
		
		$f\hspace{2pt}^*$ is the inverse image functor, mapping any subset $Y\hspace{1pt}' \subseteq Y$ to $\{x \in X \mid f\hspace{2pt}(x) \in Y\hspace{1pt}'\} \subseteq Y$, which clearly preserves finite suprema (unions). 
		
		$\exists_f$ is the image functor, which maps any subset $X\hspace{1pt}' \subseteq X$ to $\{f\hspace{2pt}(x) \mid x \in X\hspace{1pt}'\} \subseteq Y $. 
		
		$\forall_f$ is the functor mapping any subset $X\hspace{1pt}' \subseteq X$ to the set abstract $\{y \in Y \mid \forall x \in X .(f\hspace{2pt}(x) = y \rightarrow x \in X\hspace{1pt}')\} \subseteq Y$. 
		
		Let $X\hspace{1pt}' \subseteq X$ and $Y\hspace{1pt}' \subseteq Y$. It is easily seen that $\exists_f\hspace{2pt}(X\hspace{1pt}') \subseteq Y\hspace{1pt}' \iff X\hspace{1pt}' \subseteq f\hspace{2pt}^*(Y\hspace{1pt}')$, i.e. $\exists_f \dashv f\hspace{2pt}^*$. It is also easily seen that $f\hspace{2pt}^*(Y\hspace{1pt}') \subseteq X\hspace{1pt}' \iff Y\hspace{1pt}' \subseteq \forall_f\hspace{2pt}(X\hspace{1pt}')$, i.e. $f\hspace{2pt}^* \dashv \forall_f $.
		\end{enumerate}
		
	\item In the verification of the HC axioms above, when the objects and morphisms are in $\mathbf{N}$, the same sets are constructed regardless if the HC axioms are verified for $\mathbf{M}$ or $\mathbf{N}$. It follows that $\mathbf{N}$ is a Heyting subcategory of $\mathbf{M}$.
	\item In both $\mathbf{M}$ and $\mathbf{N}$, a morphism is an isomorphism iff it is bijective. Hence, $\mathbf{N}$ is a conservative subcategory of $\mathbf{M}$.
	\item Interpret $U$ as $V$, the set $\{x \mid x = x\}$, which is a superset of every set, and hence a universal object in $\mathbf{N}$.
	\item For any object $x$ and morphism $m : A \rightarrow B$ of $\mathbf{M}$, interpret $\mathbf{T}_\mathrm{Ob}(x)$ as $\{\{u\} \mid u \in x\}$; and interpret $\mathbf{T}_\mathrm{Mor}(m)$ as ``the class coding the morphism $(\{x\} \mapsto \{m(x)\}) : \mathbf{T}A \rightarrow \mathbf{T}B$''. Since these formulae stratified, $\mathbf{T}$ restricts appropriately to $\mathbf{N}$. It is easily verified that the interpreted axioms of a functor hold.
	\item For each object $x$ in $\mathbf{M}$, interpret $\iota_x$ as the code of the morphism $(z \mapsto \{z\}) : x \rightarrow \mathbf{T}(x)$, which is a class. Since the inverse of $\iota_x$ is similarly interpretable, we obtain that the interpretation of $\iota_x$ is an isomorphism in the category theoretic sense. That $\iota$ is a natural isomorphism on $\mathbf{M}$ is clear from its definition and the definition of $\mathbf{T}$. (A word of caution: $\iota_x$ is not generally a set even if $x$ is, in fact $\iota_V$ is a proper class.)
	\item For each object $x$ in $\mathbf{N}$, interpret $\mathbf{P}_\mathrm{Ob}(x)$ as $\pow x$. For each morphism $m : A \rightarrow B$ in $\mathbf{N}$, interpret $\mathbf{P}_\mathrm{Mor}(m)$ as ``the set coding the morphism $(x \mapsto \{m(z) \mid z \in x\}) : \pow A \rightarrow \pow B$''. It is easily seen that this makes $\mathbf{P}$ an endofunctor on $\mathbf{N}$.
	\item Let $x$ be an object in $\mathbf{N}$. Note that $\mathbf{P}\mathbf{T}x = \pow \{\{z\} \mid z \in x\}$. Interpret $\mu_x : \mathbf{P}\mathbf{T}x \rightarrow \mathbf{T}\mathbf{P}x$ by the set coding the morphism $(u \mapsto \{\cup u\}) : \pow \{\{z\} \mid z \in y\} \rightarrow \{\{v\} \mid v \in \pow y\}$. Union and singleton are defined by stratified formulae. Because the union operation lowers type by one and the singleton operation raises type by one, argument and value are type-level in the formula defining $\mu_x$, so $\mu_x$ is coded by a set and is therefore a morphism in $\mathbf{N}$. It is easily seen from the constructions of $\mathbf{T}$, $\mathbf{P}$ and $\mu$, that $\mu$ is a natural isomorphism.
	\item Define $x \subseteq^{\mathbf{T}} y$ set theoretically by $\exists u . (x = \{u\} \wedge u \in y)$. For each object $A$ of $\mathbf{N}$, interpret $\subseteq^{\mathbf{T}}_A \hookrightarrow \mathbf{T} A \times \mathbf{P} A$ as the set coding the inclusion function of $\{\langle x, y \rangle \in \mathbf{T}A \times \mathbf{P}A \mid x \subseteq^{\mathbf{T}} y\} \subseteq \mathbf{T}A \times \mathbf{P}A$.
	\item We proceed to verify that $\mathbf{T}$, $\mathbf{P}$ and $\subseteq^\mathbf{T}$ satisfy the property (\ref{PT}). Suppose that $r : R \rightarrowtail\mathbf{T} A \times B$ is a mono in $\mathbf N$. Let $\chi : B \rightarrow \pow A$ code the function $(y \mapsto \big\{u \mid \exists c \in R . r(c) = \langle \{u\}, y \rangle \big\})$. Since this is a stratified definition, where argument and value have equal type, $\chi$ is a morphism in $\mathbf{N}$. The proof that $\chi$ is the unique morphism making (\ref{PT}) a pullback in $\mathbf{N}$ is just like the standard proof in conventional set theory; it proceeds as follows. We may assume that $R \subseteq A \times B$ and $r$ is the inclusion function. Then $\chi$ is $(y \mapsto \{u \mid \{u\} R y \})$. For the top arrow in (\ref{PT}) we choose $(\id \times \chi)\restriction_R$. Since $\forall \langle \{u\}, y\rangle \in R . \{u\} \subseteq^{\mathbf{T}} \chi(y)$, (\ref{PT}) commutes. 
	
	For the universal pullback property: Suppose that $\langle f, g\rangle : Q \rightarrow \mathbf{T}A \times B$ and $\langle d, e\rangle : Q \rightarrow \subseteq^{\mathbf{T}}_A$ are morphisms making the diagram commute in $\mathbf{N}$. Let $q \in Q$ be arbitrary. Then $f\hspace{2pt}(q) = d(q)$, $\chi(g(q)) = e(q)$ and $d(q) \subseteq^\mathbf{T} e(q)$, so $f\hspace{2pt}(q) \subseteq^{\mathbf{T}} \chi(g(q))$, whence by definition of $\chi$ we have $f\hspace{2pt}(q) R g(q)$. Thus, $(q \mapsto \langle f\hspace{2pt}(q), g(q)\rangle)$ defines the unique morphism from $Q$ to $R$ in $\mathbf{N}$, witnessing the universal pullback property. 
	
	It remains to show that if $\chi\hspace{1pt}'$ is a morphism in $\mathbf{N}$ that (in place of $\chi$) makes (\ref{PT}) a pullback in $\mathbf{N}$, then $\chi\hspace{1pt}' = \chi$. Let $\chi\hspace{1pt}'$ be such a morphism, and let $u \in A$ and $y \in B$. Since $\{u\} R y \Leftrightarrow \{u\} \subseteq^{\mathbf{T}} \chi(y)$, it suffices to show that $\{u\} R y \Leftrightarrow \{u\} \subseteq^{\mathbf{T}} \chi\hspace{1pt}'(y)$. By commutativity $\{u\} R y \Rightarrow \{u\} \subseteq^{\mathbf{T}} \chi\hspace{1pt}'(y)$. Conversely, applying the universal pullback property to the inclusion function $\{\langle \{u\}, y\rangle\} \hookrightarrow \mathbf{T}A \times B$, we find that $\{u\} \subseteq^{\mathbf{T}} \chi\hspace{1pt}'(y) \Rightarrow \{u\} R y$.
	\end{enumerate}
This completes the interpretation of $\mathrm{IMLU}_\mathsf{Cat}$ in $\mathrm{IMLU}_\mathsf{Class}$. For $\mathrm{MLU}$, simply observe that $\mathrm{MLU}_\mathsf{Class} \vdash \forall X . \forall X\hspace{1pt}' \subseteq X . X\hspace{1pt}' \cup (X - X\hspace{1pt}') = X$, so each $\mathrm{Sub}_X$ is Boolean. For $\mathrm{(I)ML}$,  the fact that $V = \pow V$ ensures that the interpretation of $U \cong \mathbf{P}U$ holds.
\end{proof}

\section{Interpretation of the $\mathsf{Set}$-theories in the $\mathsf{Cat}$-theories} \label{interpret set in cat}

For the rest of the paper, fix an $\mathrm{IMLU}$-category $( \mathbf{M}, \mathbf{N})$ -- along with $U$ in $\mathbf{N}$, $\mathbf{T}: \mathbf{M} \rightarrow \mathbf{M}$ (restricting to an endofunctor of $\mathbf{N}$), $\iota : \mathbf{id} \xrightarrow{\sim} \mathbf{T}$ on $\mathbf{M}$, $\mathbf{P}: \mathbf{N} \rightarrow \mathbf{N}$, $\mu : \mathbf{P}\circ\mathbf{T} \xrightarrow{\sim} \mathbf{T}\circ\mathbf{P}$, and $\subseteq^{\mathbf{T}}_X \rightarrowtail \mathbf{T} X \times \mathbf{P}X$ (for each object $X$ in $\mathbf{N}$) -- all satisfying the conditions in Definition \ref{IMLU_CAT}. Moreover, fix an object $\mathbf{1}$ which is terminal in both $\mathbf{M}$ and $\mathbf{N}$ and fix a product functor $\times$ on $\mathbf{M}$ which restricts to a product functor on $\mathbf{N}$. This can be done since $\mathbf{N}$ is a Heyting subcategory of $\mathbf{M}$. Given an $n \in \mathbb{N}$ and a product $P$ of $n$ objects, the $i$-th projection morphism, for $i = 1, \dots , n$, is denoted $\pi_P^i$.

In this section, we shall establish that 
$$\mathrm{Con}(\mathrm{(I)ML(U)}_\mathsf{Cat}) \Rightarrow \mathrm{Con}(\mathrm{(I)NF(U)}_\mathsf{Set}).$$ 
We do so by proving that the axioms of $\mathrm{(I)NF(U)}_\mathsf{Set}$ can be interpreted in the internal language of $\mathrm{(I)ML(U)}_\mathsf{Cat}$. In particular, we construct a structure in the categorical semantics of $\mathbf{M}$ which satisfies the axioms of $\mathrm{(I)NF(U)}_\mathsf{Set}$. The variation between the intuitionistic and the classical case is handled by Theorem \ref{Completeness}, so we will concentrate on proving 
$$\mathrm{Con}(\mathrm{IMLU}_\mathsf{Cat}) \Rightarrow \mathrm{Con}(\mathrm{INFU}_\mathsf{Set}),$$ 
and $\mathrm{Con}(\mathrm{MLU}_\mathsf{Cat}) \Rightarrow \mathrm{Con}(\mathrm{NFU}_\mathsf{Set})$ is thereby obtained as well, simply by assuming that $( \mathbf{M}, \mathbf{N})$ is an $\mathrm{MLU}$-category. By Lemma \ref{SforNF} below this also establishes $\mathrm{Con}(\mathrm{(I)ML}_\mathsf{Cat}) \Rightarrow \mathrm{Con}(\mathrm{(I)NF}_\mathsf{Set})$.

\begin{constr} \label{InConstruction}
For each object $A$ of $\mathbf{N}$ let $\in_A$, along with $m_{\in_A}$, be this pullback in $\mathbf{M}$:
$$
\begin{tikzcd}[ampersand replacement=\&, column sep=small]
{\in_A} \ar[rr, "{\sim}"] \ar[d, rightarrowtail, "{m_{\in_A}}"] \&\& {\subseteq^{\mathbf{T}}_A} \ar[d, rightarrowtail] \\
{A \times \mathbf{P} A} \ar[rr, "{\sim}", "{\iota \times \id}"'] \&\& {\mathbf{T} A \times \mathbf{P} A}
\end{tikzcd}
$$
\end{constr}

In order to avoid confusing the $\in_A$ defined above with the membership symbol of $\mathcal{L}_\mathsf{Set}$, the latter is replaced by the symbol $\epsin$.

\begin{constr} \label{StructureConstruction}
This $\mathcal{L}_\mathsf{Set}$-structure, in the categorical semantics of $\mathbf{M}$, is denoted $\mathcal{U}$:
	\begin{enumerate}
	\item The single sort of $\mathcal{L}_\mathsf{Set}$ is assigned to the universal object $U$ of $\mathbf{N}$.
	\item Fix a mono $m_S : \mathbf{P}U \rightarrowtail U$ in $\mathbf{N}$. The sethood predicate symbol $S$ is identified with the predicate symbol $m_S$ in $\mathcal{S}^\mathbf{M}_\mathbf{N}$, and is assigned to the subobject of $U$ determined by $m_S$.
	\item Fix the mono $m_{\epsin} =_{\mathrm{df}} (\id_U \times m_S) \circ m_{\in_U} : \in_U \rightarrowtail U \times \mathbf{P} U \rightarrowtail U \times U.$ The membership symbol $\epsin$ is identified with the symbol $m_{\epsin}$ in $\mathcal{S}^\mathbf{M}_\mathbf{M}$, and is assigned to the subobject of $U \times U$ determined by $m_{\epsin}$.
	\item Fix a mono $m_P : U \times U \rightarrowtail U$ in $\mathbf{N}$. The function symbol $\langle -, -\rangle$ is identified with the symbol $m_P$ in $\mathcal{S}^\mathbf{M}_\mathbf{N}$ and is assigned to the subobject of $U$ determined by $m_P$.
	\end{enumerate}

By the identifications of symbols, the signature of $\mathcal{L}_\mathsf{Set}$ is a subsignature of $\mathcal{S}^\mathbf{M}_\mathbf{M}$.
\end{constr}

We will usually omit subscripts such as in $\in_A$ and $\subseteq^{\mathbf{T}}_A$, as they tend to be obvious. Similarly, sort declarations are sometimes omitted when considering formulae of the internal language. Note that the symbol $\epsin$ in $\mathcal{L}_\mathsf{Set}$ is interpreted by the subobject of $U \times U$ determined by $m_{\epsin}$, not by the subobject of $U \times \mathbf{P}U$ determined by $m_{\in_U}$. In the categorical setting it tends to be more natural to have a membership relation of sort $A \times \mathbf{P}A$ for each object $A$, while in the set-theoretical setting it tends to be more natural to have just one sort, say Universe, and just one membership relation of sort Universe $\times$ Universe.

To prove $\mathrm{Con}(\mathrm{(I)ML(U)}_\mathsf{Cat}) \Rightarrow \mathrm{Con}(\mathrm{(I)NF(U)}_\mathsf{Set})$, we need to establish that $\mathcal{U}$ satisfies Axioms \ref{SCAx}. $\mathcal{U} \models \phi$ is the statement that the $\mathcal{L}_\mathsf{Set}$-structure $\mathcal{U}$ satisfies $\phi \in \mathcal{L}_\mathsf{Set}$, in the categorical semantics of $\mathbf{M}$. For the major axioms, Extensionality and Stratified Comprehension, we will first prove the more general (and more naturally categorical) statements in terms of the $\in_A$, and second obtain the required statements about $\epsin$ as corollaries. The general results will be stated in the form $\mathbf{M} \models \phi$, where $\phi$ is a formula in the language of $\mathcal{S}^M_M$ or some subsignature of it. In particular, the subsignature $\mathcal{S}^\mathbf{M}_\mathbf{N}$ is of interest. Since $\mathbf{N}$ is a Heyting subcategory of $\mathbf{M}$, if $\phi(\vec{x})$ is an $\mathcal{S}^\mathbf{M}_\mathbf{N}$-formula (with $\vec{x} : X^n$, for some $X$ in $\mathbf{N}$ and $n \in \mathbb{N}$), then $\llbracket \vec{x} : X^n \mid \phi(\vec{x}) \rrbracket$ is assigned to the same subobject of $X^n$ by the natural $\mathcal{S}^\mathbf{M}_\mathbf{M}$-structure as by the natural $\mathcal{S}^\mathbf{M}_\mathbf{N}$-structure. Therefore, we do not need to specify which of these structures is used when referring to a subobject by such an expression.

The following proposition is the expression of Construction \ref{InConstruction} in the categorical semantics.

\begin{prop}\label{membership as iota subset}
Let $X$ be an object of $\mathbf{N}$. 
$$\mathbf{M} \models \forall x : X . \forall y : \mathbf{P} X . ( x \in y \leftrightarrow \iota x \subseteq^{\mathbf{T}} y).$$
\end{prop}
\begin{proof}
$\llbracket x, y \mid \iota x \subseteq^{\mathbf{T}} y \rrbracket = (\iota \times \id)^*\llbracket u, y \mid u \subseteq^{\mathbf{T}} y \rrbracket = \llbracket x, y \mid x \in y \rrbracket$. 
\end{proof}

Let us start the proof of $\mathcal{U} \models \mathrm{INFU}_\mathsf{Set}$ with the easy axioms of Sethood and Ordered Pair.

\begin{prop}[Sethood] \label{Sethood}
$\mathcal{U} \models \forall z . \forall x . (z \epsin x \rightarrow S(x))$
\end{prop}
\begin{proof}
By construction of $m_{\epsin}$, $\llbracket z, x : U \mid z \epsin x \rrbracket \leq_{U \times U} U \times \mathbf{P}U$, and by construction of  $m_S$, $U \times \mathbf{P}U \cong_{U \times U} \llbracket z, x : U \mid S(x)\rrbracket$, so 
$$\llbracket z, x : U \mid z \epsin x \rrbracket \leq_{U \times U} \llbracket z, x : U \mid S(x)\rrbracket,$$ 
as desired.
\end{proof}

\begin{prop}[Ordered Pair] \label{OrderedPair}
$\mathcal{U} \models \forall x, x\hspace{1pt}', y, y\hspace{1pt}' . \big( \langle x, y \rangle = \langle x\hspace{1pt}', y\hspace{1pt}' \rangle \rightarrow ( x = x\hspace{1pt}' \wedge y = y\hspace{1pt}' ) \big)$
\end{prop}
\begin{proof}
Let $\langle a, a\hspace{1pt}', b, b\hspace{1pt}' \rangle$ be a mono with co-domain $U^4$, representing 
\[
\llbracket x, x\hspace{1pt}', y, y\hspace{1pt}' : U \mid \forall x, x\hspace{1pt}', y, y\hspace{1pt}' . \big( \langle x, y \rangle = \langle x\hspace{1pt}', y\hspace{1pt}' \rangle \rightarrow ( x = x\hspace{1pt}' \wedge y = y\hspace{1pt}' ) \big) \rrbracket .
\]
We need to derive $\langle a, b \rangle = \langle a\hspace{1pt}', b\hspace{1pt}' \rangle$ from the assumption $m_P \circ \langle a, b \rangle = m_P \circ \langle a\hspace{1pt}', b\hspace{1pt}' \rangle$. But this follows immediately from that $m_P$ is monic.
\end{proof}

The following Lemma yields $\mathrm{Con}(\mathrm{(I)ML}_\mathsf{Cat}) \Rightarrow \mathrm{Con}(\mathrm{(I)NF}_\mathsf{Set})$ for free, if we successfully prove that $\mathcal{U} \models \mathrm{(I)NFU_\mathsf{Set}}$.

\begin{lemma}[$\mathrm{(I)NF}$ for free]\label{SforNF}
If $U \cong \mathbf{P}U$, then we can choose $m_S : \mathbf{P}U \rightarrowtail U$ (i.e. the interpretation of the predicate symbol $S$) to be an isomorphism. If so, then $\mathcal{U} \models \forall x . S(x)$.
\end{lemma}
\begin{proof}
Since $m_S$ is an isomorphism, $m_s : \mathbf{P}U \rightarrowtail U$ and $\id : U \rightarrowtail U$ represent the same subobject of $U$.
\end{proof}

Note that we do not need $U = \mathbf{P}U$ for this result; $U \cong \mathbf{P}U$ suffices. This means that our results will actually give us that $\mathrm{(I)NF_\mathsf{Set}}$ is equiconsistent with $\mathrm{(I)NFU_\mathsf{Set}} + \big( |V| = |\mathcal P(V)| \big)$, with essentially no extra work. See Corollary \ref{ConFewAtoms} below. This result has been proved previously in \cite{Cra00} using the conventional set-theoretical semantics. In the present categorical setting, this result is transparently immediate.

\begin{prop}[Extensionality]\label{Extensionality}
Let $Z$ be an object of $\mathbf{N}$. 
$$\mathbf{M} \models \forall x : \mathbf{P} Z . \forall y : \mathbf{P} Z . \big[ \big( \forall z : Z . (z \in x \leftrightarrow z \in y) \big) \rightarrow x = y \big].$$
\end{prop}
\begin{proof}
We use the fact that $\mathbf{N}$ is a Heyting subcategory of $\mathbf{M}$. By Proposition \ref{membership as iota subset}, it suffices to establish that in $\mathbf{N}$:
\[ \llbracket x : \mathbf{P} Z, y : \mathbf{P} Z \mid \forall z : \mathbf{T}Z . (z \subseteq^{\mathbf{T}} x \leftrightarrow z \subseteq^{\mathbf{T}} y) \rrbracket \leq \llbracket x, y \mid x = y \rrbracket .\]
Let $\langle a, b \rangle : E \rightarrowtail \mathbf{P} Z \times \mathbf{P} Z$ represent $\llbracket x, y \mid \forall z . (z \subseteq^{\mathbf{T}} x \leftrightarrow z \subseteq^{\mathbf{T}} y) \rrbracket$. We need to show that $a = b$.

Consider $\llbracket w, u, v \mid w \subseteq^{\mathbf{T}} u \rrbracket$ and $\llbracket w, u, v \mid w \subseteq^{\mathbf{T}} v \rrbracket$ as subobjects of $\mathbf{T} Z \times \mathbf{P} Z \times \mathbf{P} Z$. We calculate their pullbacks along $\id \times \langle a, b \rangle$ to be equal subobjects of $\llbracket x, y \mid \forall z . (z \subseteq^{\mathbf{T}} x \leftrightarrow z \subseteq^{\mathbf{T}} y) \rrbracket$:
\[ 
\begin{split}
& (\id \times \langle a, b \rangle)^* \llbracket w, u, v \mid w \subseteq^{\mathbf{T}} u \rrbracket \\
= & \llbracket w, u, v \mid w \subseteq^{\mathbf{T}} u \wedge \forall z . (z \subseteq^{\mathbf{T}} u \leftrightarrow z \subseteq^{\mathbf{T}} v) \rrbracket \\
= & \llbracket w, u, v \mid w \subseteq^{\mathbf{T}} u \wedge w \subseteq^{\mathbf{T}} v \wedge \forall z . (z \subseteq^{\mathbf{T}} u \leftrightarrow z \subseteq^{\mathbf{T}} v) \rrbracket \\
= & \llbracket w, u, v \mid w \subseteq^{\mathbf{T}} v \wedge \forall z . (z \subseteq^{\mathbf{T}} u \leftrightarrow z \subseteq^{\mathbf{T}} v) \rrbracket \\
= & (\id \times \langle a, b \rangle)^*\llbracket w, u, v \mid w \subseteq^{\mathbf{T}} v \rrbracket 
\end{split}
\]

From inspection of the chain of pullbacks
\[
\begin{tikzcd}[ampersand replacement=\&, column sep=small]
{(\id \times \langle a, b \rangle)^* \llbracket w, u, v \mid w \subseteq^{\mathbf{T}} u \rrbracket} \ar[rr, rightarrowtail] \ar[d, rightarrowtail, "f"'] \&\&
{\llbracket w, u, v \mid w \subseteq^{\mathbf{T}} u \rrbracket} \ar[rr] \ar[d, rightarrowtail] \&\&
{\llbracket w, t \mid w \subseteq^{\mathbf{T}} t \rrbracket} \ar[d, rightarrowtail]
\\
{\llbracket z, x, y \mid \forall z . (z \subseteq^{\mathbf{T}} x \leftrightarrow z \subseteq^{\mathbf{T}} y) \rrbracket} \ar[rr, rightarrowtail, "{\id \times \langle a, b \rangle}"] \ar[rrrr, bend right, "{\id \times a}"'] \&\&
{\mathbf{T} Z \times \mathbf{P} Z \times \mathbf{P} Z} \ar[rr, "{\langle \pi^1, \pi^2 \rangle}" ] \&\& {\mathbf{T} Z \times \mathbf{P} Z} ,
\end{tikzcd}
\]
it is evident that
\[
(\id \times a)^* \big( \llbracket w, t \mid w \subseteq^{\mathbf{T}} t \rrbracket) = (\id \times \langle a, b \rangle)^* \llbracket w, u, v \mid w \subseteq^{\mathbf{T}} u \rrbracket .
\]
Similarly,
\[
(\id \times b)^* \big( \llbracket w, t \mid w \subseteq^{\mathbf{T}} t \rrbracket) = (\id \times \langle a, b \rangle)^* \llbracket w, u, v \mid w \subseteq^{\mathbf{T}} v \rrbracket .
\]
So
\[
(\id \times a)^* \big( \llbracket w, t \mid w \subseteq^{\mathbf{T}} t \rrbracket) = (\id \times b)^* \big( \llbracket w, t \mid w \subseteq^{\mathbf{T}} t \rrbracket) ,
\]
and $f$ represents them as a subobject of $\llbracket z, x, y \mid \forall z . (z \subseteq^{\mathbf{T}} x \leftrightarrow z \subseteq^{\mathbf{T}} y) \rrbracket$.

By uniqueness of $\chi$ in (\ref{PT}), we conclude that $a = b$.
\end{proof}

\begin{cor} \label{ExtCor}
$\mathcal{U} \models \mathrm{Ext}_S$
\end{cor}
\begin{proof}
By Proposition \ref{Extensionality}, 
$$\mathbf{M} \models \forall x : \mathbf{P} U . \forall y : \mathbf{P} U . \big[ \big( \forall z : U . (z \in x \leftrightarrow z \in y) \big) \rightarrow x = y \big].$$ 
Now, by routine categorical semantics,
$$
\begin{array}{rl}
& \mathbf{M} \models \forall x, y : \mathbf{P} U . \big[ \big( \forall z : U . (z \in x \leftrightarrow z \in y) \big) \rightarrow x = y \big] \\
\iff & \mathbf{M} \models \forall  x, y : \mathbf{P} U . \big[ \big( \forall z : U . (z \epsin m_S(x) \leftrightarrow z \epsin m_S(y)) \big) \rightarrow x = y \big] \\
\iff & \mathbf{M} \models \forall x\hspace{1pt}', y\hspace{1pt}' : U . \big[ \big(S(x\hspace{1pt}') \wedge S(y\hspace{1pt}')\big) \rightarrow \\
& \big( ( \forall z : U . (z \epsin x\hspace{1pt}' \leftrightarrow z \epsin y\hspace{1pt}') ) \rightarrow x\hspace{1pt}' = y\hspace{1pt}' \big) \big] . \\
\end{array}
$$
So $\mathcal{U} \models \mathrm{Ext}_S$.
\end{proof}

The only axiom of $\mathrm{INFU}_\mathsf{Set}$ left to validate is $\mathrm{SC}_S$ (i.e. Stratified Comprehension). In order to approach this, we first need to construct some signatures and define stratification for an appropriate internal language:

\begin{dfn} \label{MoreSignatures}
Let $\mathcal{S}^M_{N, \in}$ be the subsignature of $\mathcal{S}^M_M$ containing $\mathcal{S}^M_N$ and the relation symbol $\in_A$, which is identified with $m_{\in_A}$, for each object $A$ in $\mathbf{N}$. 

Stratification in the language of $\mathcal{S}^M_{N, \in}$ is defined analogously as in Definition \ref{DefStrat}. A stratification function $s$, of an $\mathcal{S}^M_{N, \in}$-formula $\phi$, is an assignment of a {\em type} in $\mathbb{N}$ to each term in $\phi$, subject to the following conditions (where $\equiv$ is syntactic equality; $n \in \mathbb{N}$; $u, v, w, w_1, \dots, w_n$ are $\mathcal{S}^M_{N, \in}$-terms in $\phi$; $\theta$ is an atomic subformula of $\phi$; $R$ is a relation symbol in $\mathcal{S}^M_N$ which is not equal to $\in_X$ for any $X$ in $\mathbf{N}$; and $A$ is an object in $\mathbf{N}$):
\begin{enumerate}[(i)]
\item if $u \equiv v(w_1, \dots, w_n)$, then $s(u) = s(w_1) = \dots = s(w_n)$,
\item if $\theta \equiv R(w_1, \dots, w_n)$, then $s(w_1) = \dots = s(w_n)$,
\item if $\theta \equiv (u \in_A w)$, then $s(u) + 1 = s(w)$,
\end{enumerate}

It can easily be seen that every stratifiable formula $\phi$ has a {\em minimal} stratification $s_\phi$, in the sense that for every stratification $s$ of $\phi$ and for every term $t$ in $\phi$, $s_\phi(t) \leq s(t)$. Moreover, the minimal stratification, $s_\phi$, is determined by the restriction of $s_\phi$ to the set of variables in $\phi$.

Let $\mathcal{S}^M_{N, \iota}$ be the subsignature of $\mathcal{S}^M_M$ containing $\mathcal{S}^M_N$ and the function symbol $\iota_A$ for each object $A$ in $\mathbf{N}$.

If $\phi$ is a formula in either of these languages, then $\mathbf{M} \models \phi$ is to be understood as satisfaction in the natural $\mathcal{S}^\mathbf{M}_\mathbf{M}$-structure.
\end{dfn}

We start by verifying a form of comprehension for $\mathcal{S}^\mathbf{M}_\mathbf{N}$-formulae:

\begin{prop}\label{ComprehensionN}
If $\phi(w, y)$ is an $\mathcal{S}^\mathbf{M}_\mathbf{N}$-formula, with context $w : \mathbf{T}Z, y : Y$ for $Z, Y$ in $\mathbf N$, and in which $x$ is not free, then
$$\mathbf{M} \models \forall y : Y . \exists x : \mathbf{P} Z . \forall w : \mathbf{T} Z . (w \subseteq^{\mathbf{T}} x \leftrightarrow \phi(w, y)).$$
\end{prop}
\begin{proof}
This is a familiar property of power objects. Considering this instance of (PT) in $\mathbf{N}$:
\begin{equation}
\begin{tikzcd}[ampersand replacement=\&, column sep=small]
\llbracket w : \mathbf{T} Z, y : Y \mid \phi(w, y)\rrbracket \ar[rr] \ar[d, rightarrowtail] \&\& \subseteq^{\mathbf{T}}_{Z} \ar[d, rightarrowtail] \\
\mathbf{T} Z \times Y \ar[rr, "{\id \times \chi}"] \&\& \mathbf{T} Z \times \mathbf{P} Z
\end{tikzcd}
\end{equation}
This pullback along $\id \times \chi$ can be expressed as $\llbracket w : \mathbf{T} Z, y : Y \mid w \subseteq^{\mathbf{T}}_{Z} \chi(y)\rrbracket$ in $\mathbf{N}$. So since $\mathbf{N}$ is a Heyting subcategory of $\mathbf{M}$,
$$\mathbf{M} \models \forall y : Y . \forall w : \mathbf{T} Z . (w \subseteq^{\mathbf{T}} \chi(y) \leftrightarrow \phi(w, y)) \text{, and}$$
$$\mathbf{M} \models \forall y : Y . \exists x : \mathbf{P}Z . \forall w : \mathbf{T} Z . (w \subseteq^{\mathbf{T}} x \leftrightarrow \phi(w, y)),$$
as desired.
\end{proof}

To obtain stratified comprehension for $\mathcal{S}^\mathbf{M}_{\mathbf{N}, \in}$-formulae, we need to establish certain coherence conditions. The facts that $\mathbf{N}$ is a Heyting subcategory of $\mathbf{M}$ and that $\mathbf{T}$ preserves limits as an endofunctor of $\mathbf{N}$ (see Proposition \ref{strenghten def} (\ref{strengthen T limits})), enable us to prove that certain morphisms constructed in $\mathbf{M}$ also exist in $\mathbf{N}$, as in the lemmata below. This is useful when applying (\ref{PT}), since the relation $R$ is required to be in $\mathbf{N}$ (see Definition \ref{IMLU_CAT}).

\begin{lemma} \label{TandProductsCommute}
Let $n \in \mathbb{N}$.
\begin{enumerate}[{\normalfont (a)}]
\item $\iota_1 : 1 \rightarrow \mathbf{T}1$ is an isomorphism in $\mathbf{N}$. 
\item For any $A, B$ of $\mathbf{N}$, $(\iota_{A} \times \iota_{B}) \circ \iota_{A \times B}^{-1} : \mathbf{T}(A \times B) \xrightarrow{\sim} A \times B \xrightarrow{\sim} \mathbf{T} A \times \mathbf{T} B$ is an isomorphism in $\mathbf{N}$.
\item For any $A_1, \dots, A_n$ in $\mathbf{N}$, 
\begin{align*}
& (\iota_{A_1} \times \dots \times \iota_{A_n}) \circ \iota_{A_1 \times \dots \times A_n}^{-1} : \\
& \mathbf{T}(A_1 \times \dots \times A_n) \xrightarrow{\sim} A_1 \times \dots \times A_n \xrightarrow{\sim} \mathbf{T} A_1 \times \dots \times \mathbf{T} A_n
\end{align*}
is an isomorphism in $\mathbf{N}$.
\end{enumerate}
\end{lemma}
\begin{proof}
\begin{enumerate}
\item Since $\mathbf{T}$ preserves limits, $\mathbf{T}\mathbf{1}$ is terminal in $\mathbf{N}$, and since $\mathbf{N}$ is a Heyting subcategory of $\mathbf{M}$, $\mathbf{T}\mathbf{1}$ is terminal in $\mathbf{M}$ as well. So by the universal property of terminal objects, $\mathbf{1}$ and $\mathbf{T}\mathbf{1}$ are isomorphic in $\mathbf{N}$, and the isomorphisms must be $\iota_\mathbf{1}$ and $\iota^{-1}_\mathbf{1}$.
\item Since $\mathbf{T}$ preserves limits, $\mathbf{T}(A \times B)$ is a product of $\mathbf{T}A$ and $\mathbf{T}B$ in $\mathbf{N}$, and since $\mathbf{N}$ is a Heyting subcategory of $\mathbf{M}$, it is such a product in $\mathbf{M}$ as well. Now note that 
$$\pi_{\mathbf{T} A \times \mathbf{T} B}^1 \circ (\iota_A \times \iota_B) \circ \iota^{-1}_{A \times B} = \iota_A \circ \pi_{A \times B}^1 \circ \iota_{A \times B}^{-1} = \mathbf{T}\pi_{A \times B}^1,$$ 
and similarly for the second projection. The left equality is a basic fact about projection morphisms. The right equality follows from that $\iota$ is a natural isomorphism. This means that 
$$(\iota_A \times \iota_B) \circ \iota^{-1}_{A \times B} : \mathbf{T}(A \times B) \xrightarrow{\sim} \mathbf{T}A \times \mathbf{T}B$$ 
is the unique universal morphism provided by the definition of product. Hence, it is an isomorphism in $\mathbf{N}$. 
\item This follows from the two items above by induction on $n$.
\end{enumerate}
\end{proof}

\begin{lemma} \label{iotaTermConversion}
Let $n \in \mathbb{N}$. If $u : A_1 \times \dots \times A_n \rightarrow B$ is a morphism in $\mathbf{N}$, then there is a morphism $v : \mathbf{T} A_1 \times \dots \times \mathbf{T} A_n \rightarrow \mathbf{T} B$ in $\mathbf{N}$, such that 
$$\iota_B \circ u = v \circ (\iota_{A_1} \times \dots \times \iota_{A_n}).$$
\end{lemma}
\begin{proof}
Since $\iota$ is a natural transformation, 
$$\iota_B \circ u = (\mathbf{T}u) \circ \iota_{A_1 \times \dots \times A_n}.$$
Since $\iota$ is a natural isomorphism, 
$$\iota_{A_1 \times \dots \times A_n} = \iota_{A_1 \times \dots \times A_n} \circ (\iota_{A_1} \times \dots \times \iota_{A_n})^{-1} \circ (\iota_{A_1} \times \dots \times \iota_{A_n}).$$ 
Thus, by letting $v = (\mathbf{T}u) \circ \iota_{A_1 \times \dots \times A_n} \circ (\iota_{A_1} \times \dots \times \iota_{A_n})^{-1}$, the result is obtained from Lemma \ref{TandProductsCommute}. 
\end{proof}

\begin{constr} \label{iotaRelationConstruction}
Let $n \in \mathbb{N}$, and let $m_R : R \rightarrowtail A_1 \times \dots \times A_n$ be a morphism in $\mathbf{N}$ that is monic in $\mathbf{M}$; i.e. $m_R$ is a relation symbol in $\mathcal{S}^\mathbf{M}_\mathbf{N}$. Using the isomorphism obtained in Lemma \ref{TandProductsCommute}, we construct $\hat{\mathbf{T}} m_R : \hat{\mathbf{T}} R \rightarrowtail \mathbf{T} A_1 \times \dots \times \mathbf{T} A_n$ in $\mathbf{N}$ as the pullback of $\mathbf{T}m_R$ along that isomorphism:
\[
\begin{tikzcd}[ampersand replacement=\&, column sep=small]
{\hat{\mathbf{T}} R} \ar[rr, "\sim"] \ar[d, rightarrowtail, "{\hat{\mathbf{T}} m_R}"'] \&\& {\mathbf{T} R}  \ar[d, rightarrowtail, "{\mathbf{T} m_R}"] \\
{\mathbf{T} A_1 \times \dots \times \mathbf{T} A_n} \ar[rr, "\sim"] \&\& {\mathbf{T} (A_1 \times \dots \times A_n)}
\end{tikzcd}
\] 
\end{constr}
Note that the definition of $\hat{\mathbf{T}}m_R$ implicitly depends on the factorization $A_1 \times A_n$ chosen for the co-domain of $m_R$.

\begin{lemma} \label{iotaRelationConversion}
Let $m_R : R \rightarrowtail A_1 \times \dots \times A_n$ be as in Construction \ref{iotaRelationConstruction}.
\begin{align*}
\mathbf{M} \models \forall x_1 : A_1 \dots \forall x_n : A_n . \big( & (\mathbf{T} m_R)(\iota_{A_1 \times \dots \times A_n}(x_1, \dots, x_n)) \leftrightarrow \\ 
& (\hat{\mathbf{T}} m_R)(\iota_{A_1}(x_1), \dots, \iota_{A_n}(x_n)) \big).
\end{align*}
\end{lemma}
\begin{proof}
The subobjects 
\begin{align*}
P = \llbracket x_1 : A_1, \dots, x_n : A_n \mid (\mathbf{T} m_R)(\iota_{A_1 \times \dots \times A_n}(x_1, \dots, x_n)) \rrbracket \\
P' = \llbracket x_1 : A_1, \dots, x_n : A_n \mid (\hat{\mathbf{T}} m_R)(\iota_{A_1}(x_1), \dots, \iota_{A_n}(x_n)) \rrbracket
\end{align*}
of $A_1 \times \dots \times A_n$ are obtained as these pullbacks:
\[
\begin{tikzcd}[ampersand replacement=\&, column sep=small]
P \ar[rrrr, "f"] \ar[d, "g"']  \&\&\&\&  \mathbf{T}R \ar[d, "{\mathbf{T}m_R}"]  \\
{A_1 \times \dots \times A_n} \ar[rrrr, "{\iota_{A_1 \times \dots \times A_n}}"']  \&\&\&\&  {\mathbf{T}(A_1 \times \dots \times A_n)}  
\end{tikzcd}
\] 
\[
\begin{tikzcd}[ampersand replacement=\&, column sep=small]
P' \ar[rrrr, "f\hspace{2pt}'"] \ar[d, "g\hspace{1pt}'"']  \&\&\&\&  {\hat{\mathbf{T}}R} \ar[d, "{\hat{\mathbf{T}}m_R}"]  \\
{A_1 \times \dots \times A_n} \ar[rrrr, "{\iota_{A_1} \times \dots \times \iota_{A_n}}"']  \&\&\&\&  {\mathbf{T}A_1 \times \dots \times \mathbf{T}A_n}  
\end{tikzcd}
\] 
Since the bottom morphisms in both of these pullback-diagrams are isomorphisms, it follows from a basic fact about pullbacks that the top ones, $f$ and $f\hspace{2pt}'$, are also isomorphisms. So $P$ and $P'$, as subobjects of $A_1 \times \dots \times A_n$, are also represented by $\iota^{-1}_{A_1 \times \dots \times A_n} \circ \mathbf{T}m_R : \mathbf{T}R \rightarrow A_1 \times \dots \times A_n$ and $(\iota_{A_1} \times \dots \times \iota_{A_n})^{-1} \circ \hat{\mathbf{T}}m_R : \hat{\mathbf{T}}R \rightarrow A_1 \times \dots \times A_n$, respectively.

Now note that by construction of $\hat{\mathbf{T}}$, this diagram commutes:
\[
\begin{tikzcd}[ampersand replacement=\&, column sep=small]
{\hat{\mathbf{T}} R} \ar[rr, "\sim"] \ar[d, rightarrowtail, "{\hat{\mathbf{T}} m_R}"'] \&\& {\mathbf{T} R}  \ar[d, rightarrowtail, "{\mathbf{T} m_R}"] \\
{\mathbf{T} A_1 \times \dots \times \mathbf{T} A_n} \ar[rr, "\sim"] \ar[rd, "{\sim}"'] \&\& {\mathbf{T} (A_1 \times \dots \times A_n)} \ar[ld, "{\sim}"] \\
\& {A_1 \times \dots \times A_n}
\end{tikzcd}
\] 
Therefore, $\iota^{-1}_{A_1 \times \dots \times A_n} \circ \mathbf{T}m_R$ and $(\iota_{A_1} \times \dots \times \iota_{A_n})^{-1} \circ \hat{\mathbf{T}}m_R$ represent the same subobject of $A_1 \times \dots \times A_n$, as desired
\end{proof}

Let $n \in \mathbb{N}$. We recursively define iterated application of $\mathbf{P}$, $\mathbf{T}$ and $\hat{\mathbf{T}}$ in the usual way, as $\mathbf{P}^0 = \mathbf{id}_\mathbf{N}$, $\mathbf{P}^{k+1} = \mathbf{P} \circ \mathbf{P}^k$ (for $k \in \mathbb{N}$), etc. The iterated application of $\iota$ requires a special definition. We define $\iota^n_A : A \xrightarrow{\sim} \mathbf{T}^n A$ recursively by 
\begin{align*}
\iota^0_A &= \id_A, \\
\iota^{k+1}_A &= \iota_{\mathbf{T}^k A} \circ \iota^k_A : A \xrightarrow{\sim} \mathbf{T}^k A \xrightarrow{\sim} \mathbf{T}^{k+1} A \text{, where } k \in \mathbb{N}.
\end{align*}
Since $\iota$ is a natural isomorphism, we have by induction that $\iota^n : \id \xrightarrow{\sim} \mathbf{T}^n$ also is a natural isomorphism.

\begin{lemma} \label{iotaTypeIncrease}
Let $n, k \in \mathbb{N}$. Let $m_R : R \rightarrowtail A_1 \times \dots A_n$ be as in Construction \ref{iotaRelationConstruction}. 
\begin{align*}
\mathbf{M} \models \forall x_1 : A_1 \dots \forall x_n : A_n . \big( & m_R(x_1, \dots, x_n) \leftrightarrow \\
& (\hat{\mathbf{T}}^k m_R) (\iota^k x_1, \dots, \iota^k x_n)\big).
\end{align*}
\end{lemma}
\begin{proof}
Since $\iota^k : \id_\mathbf{M} \rightarrow \mathbf{T}^k$ is a natural isomorphism, 
\begin{align*}
\mathbf{M} \models \forall x_1 : A_1 \dots \forall x_n : A_n . \big( & m_R(x_1, \dots, x_n) \leftrightarrow \\
& (\mathbf{T}^k m_R) (\iota_{A_1 \times \dots \times A_n}^k (x_1, \dots, x_n))\big).
\end{align*}
By iterating Lemma \ref{iotaRelationConversion}, we obtain by induction that 
\begin{align*}
\mathbf{M} \models \forall x_1 : A_1 \dots \forall x_n : A_n . \big( & (\mathbf{T}^k m_R) (\iota_{A_1 \times \dots \times A_n}^k (x_1, \dots, x_n)) \leftrightarrow \\
& (\hat{\mathbf{T}}^k m_R) (\iota_{A_1}^k x_1, \dots, \iota_{A_n}^k x_n)\big).
\end{align*}
The result now follows by combining the two. 
\end{proof}

We shall now show that any stratified $\mathcal{S}^\mathbf{M}_{\mathbf{N}, \in}$-formula $\phi$ can be converted to an $\mathcal{S}^\mathbf{M}_\mathbf{N}$-formula $\phi^{\subseteq^{\mathbf{T}}}$, which is equivalent to $\phi$ in $\mathbf{M}$, i.e. $\mathbf{M} \models \phi \leftrightarrow \phi^{\subseteq^{\mathbf{T}}}$. 

\begin{constr} \label{StratConstruction}
Let $\phi$ be any stratified $\mathcal{S}^M_{N, \in}$-formula. Let $s_\phi$ be the minimal stratification of $\phi$, and let $\mathrm{max}_\phi$ be the maximum value attained by $s_\phi$.
	\begin{itemize}
	\item Let $\phi^\iota$ be the $\mathcal{S}^\mathbf{M}_{\mathbf{N}, \iota}$-formula obtained from $\phi$ by the construction below. We shall replace each atomic subformula $\theta$ of $\phi$ by another atomic formula which is equivalent to $\theta$ in $\mathbf{M}$. We divide the construction into two cases, depending on whether or not $\theta$ is of the form $\theta \equiv t \in_X t'$, for some $X$ in $\mathbf{N}$ and terms $t, t'$ in $\mathcal{S}^\mathbf{M}_\mathbf{N}$:
		\begin{enumerate}
		\item Suppose that $\theta$ is {\em not} of the form $\theta \equiv t \in_X t'$. Then $\theta$ is equivalent in $\mathbf{M}$ to a formula $m_R(x_1, \dots, x_n)$, where $x_1, \dots, x_n$ are variables, as such a monomorphism $m_R$ can be constructed in $\mathbf{N}$ from the interpretations of the relation-symbol and terms appearing in $\theta$. Note that by stratification, $s_\phi(x_1) = \dots = s_\phi(x_n)$. Let $k = \mathrm{max} - s_\phi(x_1)$. In $\phi$, replace $\theta$ by 
		$$(\hat{\mathbf{T}}^k m_R) (\iota^k x_1, \dots, \iota^k x_n).$$
It follows from Lemma \ref{iotaTypeIncrease} that this formula is equivalent to $\theta$ in $\mathbf{M}$.
		\item Suppose that $\theta \equiv u(x_1, \dots, x_n) \in_A v(y_1, \dots, y_{m})$, where $A$ is an object in $\mathbf{N}$, $u, v$ are terms in $\mathcal{S}^\mathbf{M}_\mathbf{N}$, and $x_1, \dots, x_n, y_1, \dots, y_{m}$ are variables. Note that by stratification, 
		$$s_\phi(u) + 1 = s_\phi(x_i) + 1 = s_\phi(v) = s_\phi(y_j),$$
		for each $1 \leq i \leq n$ and $1 \leq j \leq m$. Let $k_u = \mathrm{max}_\phi - s_\phi(u)$ and $k_v = \mathrm{max}_\phi - s_\phi(v)$, whence $k_u = k_v + 1$. By Proposition \ref{membership as iota subset}, $\theta$ is equivalent in $\mathbf{M}$ to 
		\[(\iota_A \circ u) (x_1, \dots, x_n) \subseteq^\mathbf{T}_A v(y_1, \dots, y_{m}).\] 
		Thus, by Lemma \ref{iotaTypeIncrease}, $\theta$ is equivalent in $\mathbf{M}$ to 
		\[(\iota^{k_u}_A \circ u) (x_1, \dots, x_n) (\hat{\mathbf{T}}^{k_v} \subseteq^{\mathbf{T}}) (\iota^{k_v}_{\mathbf{P}A} \circ v) (y_1, \dots, y_{m}).\]
		Now, by iterated application of Lemma \ref{iotaTermConversion}, there are morphisms $u\hspace{1pt}', v\hspace{1pt}'$ in $\mathbf{N}$, such that $\theta$ is equivalent in $\mathbf{M}$ to
		$$u\hspace{1pt}'(\iota^{k_u}(x_1), \dots, \iota^{k_u}(x_n)) (\hat{\mathbf{T}}^{k_v} \subseteq^{\mathbf{T}}) v\hspace{1pt}'(\iota^{k_v}(y_1), \dots, \iota^{k_v}(y_m)).$$ 
		Replace $\theta$ by this formula.
		\end{enumerate}
	\item Let $\phi^{\subseteq^{\mathbf{T}}}$ be the formula in the language of $\mathcal{S}^\mathbf{M}_\mathbf{N}$ obtained from $\phi^\iota$ as follows.
		\begin{itemize}
		\item Replace each term of the form $\iota^{\mathrm{max}_\phi - s_\phi(x)}(x)$ (where $x$ is a variable of sort $A$) by a fresh variable $x\hspace{1pt}'$ (of sort $\mathbf{T}^{\mathrm{max}_\phi - s_\phi(x)} A$).
		\item Replace each quantifier scope or context declaration $x : A$ by $x\hspace{1pt}' : \mathbf{T}^{\mathrm{max}_\phi - s_\phi(x)} A$.
		\end{itemize}
	\end{itemize}
\end{constr}

By construction $\phi^\iota$ is an $\mathcal{S}^\mathbf{M}_{\mathbf{N}, \iota}$-formula, which is equivalent to $\phi$ in $\mathbf{M}$. Let $x$ be an arbitrary variable in $\phi^\iota$. Note that each variable $x$ in $\phi^\iota$ occurs in a term $\iota^{\mathrm{max}_\phi - s_\phi(x)} x$; and conversely, every occurrence of $\iota$ in $\phi^\iota$ is in such a term $\iota^{\mathrm{max}_\phi - s_\phi(x)} x$, for some variable $x$. Therefore, $\phi^{\subseteq^{\mathbf{T}}}$ is an $\mathcal{S}^\mathbf{M}_\mathbf{N}$-formula. So since $\iota^{\mathrm{max}_\phi - s_\phi(x)} : A \rightarrow \mathbf{T}^{\mathrm{max}_\phi - s_\phi(x)} A$ is an isomorphism, for each variable $x : A$ in $\phi^\iota$, we have that $\phi^{\subseteq^{\mathbf{T}}}$ is equivalent to $\phi^\iota$. We record these findings as a lemma:

\begin{lemma}\label{StratConvert}
If $\phi$ is a stratified $\mathcal{S}^\mathbf{M}_{\mathbf{N}, \in}$-formula, then 
$$\mathbf{M} \models \phi \Leftrightarrow \mathbf{M} \models \phi^\iota \Leftrightarrow \mathbf{M} \models \phi^{\subseteq^{\mathbf{T}}},$$ 
where $\phi^\iota$ is an $\mathcal{S}^\mathbf{M}_{\mathbf{N}, \iota}$-formula and $\phi^{\subseteq^{\mathbf{T}}}$ is an $\mathcal{S}^\mathbf{M}_\mathbf{N}$-formula.
\end{lemma}

\begin{prop}[Stratified Comprehension]\label{ComprehensionStratified}
For every stratified $\mathcal{S}^\mathbf{M}_{\mathbf{N}, \in}$-formula $\phi(z)$, where $z : Z$ for some $Z$ in $\mathbf M$,
$$\mathbf{M} \models \exists x : \mathbf{P} Z . \forall z : Z . (z \in x \leftrightarrow \phi(z)).$$
\end{prop}
\begin{proof}
By Lemma \ref{StratConvert}, we have
\begin{align*}
\mathbf{M} \models \hspace{1pt} & \exists x : \mathbf{P} Z . \forall z : Z . (z \in_Z x \leftrightarrow \phi(z, y)) \\
\iff \mathbf{M} \models \hspace{1pt} & \exists x\hspace{1pt}' : \mathbf{T}^k \mathbf{P} Z . \forall z\hspace{1pt}' : \mathbf{T}^{k+1} Z . (z\hspace{1pt}' (\hat{\mathbf{T}}^k \subseteq^{\mathbf{T}}_Z) x\hspace{1pt}' \leftrightarrow \phi^{\subseteq^{\mathbf{T}}}(z\hspace{1pt}')), \tag{$\dagger$}
\end{align*}
for some $k \in \mathbb{N}$, where $x\hspace{1pt}', z\hspace{1pt}'$ are fresh variables. 

In order to apply Proposition \ref{ComprehensionN}, we need to move the $\mathbf{T}$:s through the $\mathbf{P}$ and transform the $\hat{\mathbf{T}}^k \subseteq^{\mathbf{T}}_Z$ into a $\subseteq^{\mathbf{T}}_{\mathbf{T}^k Z}$. Since $\mu : \mathbf{P}\mathbf{T} \rightarrow \mathbf{T}\mathbf{P}$ is a natural isomorphism,
$$\nu =_\mathrm{df} \mathbf{T}^{k-1}(\mu_Z) \circ \mathbf{T}^{k-2}(\mu_{\mathbf{T}Z}) \circ \dots \circ \mathbf{T}(\mu_{\mathbf{T}^{k-2} Z}) \circ \mu_{\mathbf{T}^{k-1} Z} : \mathbf{P} \mathbf{T}^k Z \xrightarrow{\sim} \mathbf{T}^k \mathbf{P} Z,$$
is an isomorphism making this diagram commute:
\[
\begin{tikzcd}[ampersand replacement=\&, column sep=small]
\subseteq^{\mathbf{T}}_{\mathbf{T}^k Z} \ar[rr, "{\sim}"] \ar[d, rightarrowtail, "{m_{\subseteq^{\mathbf{T}}_{\mathbf{T}^k Z}}}"] \&\& \hat{\mathbf{T}}^k \subseteq^{\mathbf{T}}_Z \ar[d, rightarrowtail, "{\hat{\mathbf{T}}^k m_{\subseteq^{\mathbf{T}}_Z}}"] \\
\mathbf{T}^{k+1} Z \times \mathbf{P} \mathbf{T}^k Z \ar[rr, "{\sim}", "{\id \times \nu}"'] \&\& \mathbf{T}^{k+1} Z \times \mathbf{T}^k \mathbf{P} Z
\end{tikzcd} 
\]
So introducing a fresh variable $x\hspace{1pt}'' : \mathbf{P} \mathbf{T}^k Z$, $(\dagger)$ is equivalent to
$$\mathbf{M} \models \exists x\hspace{1pt}'' : \mathbf{P} \mathbf{T}^k Z . \forall z\hspace{1pt}' : \mathbf{T}^{k+1} Z . (z\hspace{1pt}' \subseteq^{\mathbf{T}}_{\mathbf{T}^k  Z} x\hspace{1pt}'' \leftrightarrow \phi^{\subseteq^{\mathbf{T}}}(z\hspace{1pt}')).$$
By Proposition \ref{ComprehensionN} we are done.
\end{proof}

\begin{cor}
$\mathcal{U} \models \mathrm{SC}_S$
\end{cor}
\begin{proof}
Let $\phi(z)$ be a stratified formula in $\mathcal{L}_\mathsf{Set}$. By Proposition \ref{ComprehensionStratified}, 
\[\mathbf{M} \models \exists x : \mathbf{P} U . \forall z : U . (z \in x \leftrightarrow \phi(z)).
\] 
Now,
$$
\begin{array}{rl}
& \llbracket \exists x : \mathbf{P} U . \forall z : U . (z \in x \leftrightarrow \phi(z)) \rrbracket \\
= & \llbracket \exists x : \mathbf{P} U . \forall z : U . (z \epsin m_S(x) \leftrightarrow \phi(z)) \rrbracket \\
= & \llbracket \exists x\hspace{1pt}' : U . \big(S(x\hspace{1pt}') \wedge \forall z : U . (z \epsin x\hspace{1pt}' \leftrightarrow \phi(z))\big) \rrbracket. \\
\end{array}
$$
So $\mathcal{U} \models \mathrm{SC}_S$.
\end{proof}

\begin{thm}\label{ConCatConSet}
$\mathcal{U} \models \mathrm{(I)NFU}$, and if $U \cong \mathbf{P}U$, then $\mathcal{U} \models \mathrm{(I)NF}$. Thus, each of $\mathrm{(I)NF(U)}_\mathsf{Set}$ is interpretable in $\mathrm{(I)ML(U)}_\mathsf{Cat}$, respectively.
\end{thm}
\begin{proof}
The cases of $\mathrm{(I)NFU}$ are settled by the results above on Sethood, Ordered Pair, Extensionality and Stratified Comprehension. The cases of $\mathrm{(I)NF}$ now follow from Lemma \ref{SforNF}.
\end{proof}

\begin{thm}
These theories are equiconsistent\footnote{But see Remark \ref{remML}.}: 
\begin{align*}
& \mathrm{(I)ML(U)_\mathsf{Class}} \\
& \mathrm{(I)ML(U)_\mathsf{Cat}} \\
& \mathrm{(I)NF(U)_\mathsf{Set}}
\end{align*}

More precisely, $\mathrm{(I)ML(U)_\mathsf{Class}}$ interprets $\mathrm{(I)ML(U)_\mathsf{Cat}}$, which in turn interprets $\mathrm{(I)NF(U)_\mathsf{Set}}$; and a model of $\mathrm{(I)ML(U)_\mathsf{Class}}$ can be constructed from a model of $\mathrm{(I)NF(U)_\mathsf{Set}}$.
\end{thm}
\begin{proof}
Combine Theorem \ref{ToposModelOfML}, Corollary \ref{ConSetConClass}, Theorem \ref{ConClassConCat} and Theorem \ref{ConCatConSet}.
\end{proof}

\begin{cor}\label{ConFewAtoms}
These theories are equiconsistent\footnote{But see Remark \ref{remML}.}:
\begin{align*}
& \mathrm{(I)NF}_\mathsf{Set} \\
& \mathrm{(I)NFU_\mathsf{Set}} + ( |V| = |\mathcal P(V)| )
\end{align*}
\end{cor}
\begin{proof}
Only $\Leftarrow$ is non-trivial. By the proofs above, 
\[
\begin{array}{cl}
& \mathrm{Con}\big(\mathrm{(I)NFU_\mathsf{Set}} + ( |V| = |\mathcal P(V)| )\big) \\
\Rightarrow & \mathrm{Con}\big(\mathrm{(I)MLU_\mathsf{Class}} + ( |V| = |\mathcal P(V)| )\big) \\
\Rightarrow & \mathrm{Con}\big(\mathrm{(I)ML_\mathsf{Cat}}\big) \Rightarrow \mathrm{Con}\big(\mathrm{(I)NF}_\mathsf{Set}\big),
\end{array}
\]
as desired.
\end{proof}

For the classical case, this is known from \cite{Cra00}, while the intuitionistic case appears to be new.

\section{The subtopos of strongly Cantorian objects}\label{subtopos}

\begin{dfn}
An object $X$ in $\mathbf{N}$ is {\em Cantorian} if $X \cong \mathbf{T}X$ in $\mathbf{N}$, and is {\em strongly Cantorian} if $\iota_X : X \xrightarrow{\sim} \mathbf{T}X$ is an isomorphism in $\mathbf{N}$. Define $\mathbf{SCan}_{( \mathbf{M}, \mathbf{N})}$ as the full subcategory of $\mathbf{N}$ on the set of strongly Cantorian objects. I.e. its objects are the strongly Cantorian objects, and its morphism are all the morphisms in $\mathbf{N}$ between such objects. When the subscript ${( \mathbf{M}, \mathbf{N})}$ is clear from the context, we may simply write $\mathbf{SCan}$.
\end{dfn}

\begin{prop}\label{SCan has limits}
$\mathbf{SCan}_{( \mathbf{M}, \mathbf{N})}$ has finite limits.
\end{prop}
\begin{proof}
Let $L$ be a limit in $\mathbf{N}$ of a finite diagram $\mathbf{D} : \mathbf{I} \rightarrow \mathbf{SCan}$. Since $\mathbf{N}$ is a Heyting subcategory of $\mathbf{M}$, $L$ is a limit of $\mathbf{D}$ in $\mathbf{M}$; and since $\mathbf{T}$ preserves limits, $\mathbf{T}L$ is a limit of $\mathbf{T} \circ \mathbf{D}$ in $\mathbf{M}$ and in $\mathbf{N}$. But also, since $\mathbf{D}$ is a diagram in $\mathbf{SCan}$, $\mathbf{T}L$ is a limit of $\mathbf{D}$, and $L$ is a limit of $\mathbf{T} \circ \mathbf{D}$, in $\mathbf{M}$ and in $\mathbf{N}$. So there are unique morphisms in $\mathbf{N}$ back and forth between $L$ and $\mathbf{T}L$ witnessing the universal property of limits. Considering these as morphisms in $\mathbf{M}$ we see that they must be $\iota_L$ and $\iota_L^{-1}$. Now $L$ is a limit in $\mathbf{SCan}$ by fullness.
\end{proof}

Before we can show that $\mathbf{SCan}_{( \mathbf{M}, \mathbf{N})}$ has power objects, we need to establish results showing that $\mathbf{SCan}_{( \mathbf{M}, \mathbf{N})}$ is a ``nice'' subcategory of $\mathbf{N}$.

\begin{cor}\label{SCan pres and refl limits}
The inclusion functor of $\mathbf{SCan}_{( \mathbf{M}, \mathbf{N})}$ into $\mathbf{N}$ preserves and reflects finite limits.
\end{cor}
\begin{proof}
To see that it reflects finite limits, simply repeat the proof of Proposition \ref{SCan has limits}. We proceed to show that it preserves finite limits.

Let $L$ be a limit in $\mathbf{SCan}$ of a finite diagram $\mathbf{D} : \mathbf{I} \rightarrow \mathbf{SCan}$. Let $L'$ be a limit of this diagram in $\mathbf{N}$. By the proof of Proposition \ref{SCan has limits}, $L'$ is also such a limit in $\mathbf{SCan}$, whence $L$ is isomorphic to $L'$ in $\mathbf{SCan}$, and in $\mathbf{N}$. So $L$ is a limit of $\mathbf{D}$ in $\mathbf{N}$ as well.
\end{proof}

\begin{cor}\label{SCan mono iff N mono}
$m : A \rightarrow B$ is monic in $\mathbf{SCan}_{( \mathbf{M}, \mathbf{N})}$ iff it is monic in $\mathbf{N}$.
\end{cor}
\begin{proof}
($\Leftarrow$) follows from that $\mathbf{SCan}$ is a subcategory of $\mathbf{N}$.

($\Rightarrow$) Assume that $m : A \rightarrow B$ is monic in $\mathbf{SCan}$. Then $A$ along with $\id_A : A \rightarrow A$ and $\id_A : A \rightarrow A$ is a pullback of $m$ and $m$ in $\mathbf{SCan}$. By Corollary \ref{SCan pres and refl limits}, this is also a pullback in $\mathbf{N}$, from which it follows that $m$ is monic in $\mathbf{N}$.
\end{proof}

\begin{prop}\label{SCan closed under monos}
If $m : A \rightarrowtail B$ is monic in $\mathbf{N}$ and $B$ is in $\mathbf{SCan}_{( \mathbf{M}, \mathbf{N})}$, then $A$ and $m : A \rightarrowtail B$ are in $\mathbf{SCan}_{( \mathbf{M}, \mathbf{N})}$.
\end{prop}
\begin{proof}
Let $m : A \rightarrowtail B$ be a mono in $\mathbf{N}$ and assume that $\iota_B$ is in $\mathbf{N}$. Let $P$ be this pullback in $\mathbf{N}$, which is also a pullback in $\mathbf{M}$ since $\mathbf{N}$ is a Heyting subcategory:
\[
\begin{tikzcd}[ampersand replacement=\&, column sep=small]
P \ar[rrr, rightarrowtail, "n"] \ar[d, rightarrowtail, "{\langle p, q\rangle}"'] \&\&\& B \ar[d, rightarrowtail, "{\langle \id_B, \iota_B \rangle}"] \\
{A \times \mathbf{T}A} \ar[rrr, rightarrowtail, "{m \times \mathbf{T}m}"'] \&\&\& {B \times \mathbf{T}B}
\end{tikzcd}
\]

We shall now establish that the following square is also a pullback in $\mathbf{M}$:
\[
\begin{tikzcd}[ampersand replacement=\&, column sep=small]
A \ar[rrr, rightarrowtail, "m"] \ar[d, rightarrowtail, "{\langle \id_A, \iota_A \rangle}"'] \&\&\& B \ar[d, rightarrowtail, "{\langle \id_B, \iota_B \rangle}"] \\
{A \times \mathbf{T}A} \ar[rrr, rightarrowtail, "{m \times \mathbf{T}m}"'] \&\&\& {B \times \mathbf{T}B}
\end{tikzcd}
\]
The square commutes since $\iota$ is a natural isomorphism. So since $P$ is a pullback, it suffices to find $f : P \rightarrow A$ such that $\langle p, q\rangle = \langle \id_A, \iota_A \rangle \circ f$ and $n = m \circ f$. Let $f\hspace{2pt}' = p$ and let $f\hspace{2pt}'' = \iota_A^{-1} \circ q$. We shall show that $f\hspace{2pt}' = f\hspace{2pt}''$ and that this is the desired $f$. By commutativity of the former square, $n = m \circ f\hspace{2pt}'$ and $\iota_B \circ n = \mathbf{T}m \circ \iota_A \circ f\hspace{2pt}''$, whence 
$$\iota_B \circ m \circ f\hspace{2pt}' = \mathbf{T}m \circ \iota_A \circ f\hspace{2pt}''.$$
Note that $\iota_B \circ m$ is monic, and since $\iota$ is a natural transformation it is equal to $\mathbf{T}m \circ \iota_A$. Hence, $f\hspace{2pt}' = f\hspace{2pt}''$. Let $f = f\hspace{2pt}' = f\hspace{2pt}''$. We have already seen that $n = m \circ f$. That $\langle p, q\rangle = \langle \id_A, \iota_A \rangle \circ f$ is immediately seen by plugging the definitions of $f\hspace{2pt}'$ and $f\hspace{2pt}''$ in place of $f$. Since $P$ is a pullback, it follows that $f$ is an isomorphism in $\mathbf{M}$ and that the latter square is a pullback.

Since $f = p$, $f$ is in $\mathbf{N}$, and since $\mathbf{N}$ is a conservative subcategory of $\mathbf{M}$, $f$ is an isomorphism in $\mathbf{N}$. Now note that $\iota_A = \iota_A \circ f \circ f^{-1} = q \circ f^{-1}$. Therefore, $\iota_A$ is in $\mathbf{N}$ and $A$ is in $\mathbf{SCan}$. So by fullness, $m : A \rightarrow B$ is in $\mathbf{SCan}$, as desired.
\end{proof}

\begin{prop}\label{SCan has power objects}
$\mathbf{SCan}_{( \mathbf{M}, \mathbf{N})}$ has power objects.
\end{prop}
\begin{proof}
Let $A$ be in $\mathbf{SCan}$. We shall show that $\mathbf{P}A$ along with $(\iota_A^{-1} \times \id_{\mathbf{P}A}) \circ m_{\subseteq^{\mathbf{T}}_A} : \hspace{1pt} \subseteq^{\mathbf{T}}_A \hspace{1pt} \rightarrowtail A \times \mathbf{P}A$ is a power object of $A$ in $\mathbf{SCan}$. In Step 1 we show that $(\iota_A^{-1} \times \id_{\mathbf{P}A}) \circ m_{\subseteq^{\mathbf{T}}_A} : \hspace{1pt} \subseteq^{\mathbf{T}}_A \hspace{1pt} \rightarrowtail A \times \mathbf{P}A$ is in $\mathbf{SCan}$, and in Step 2 we show that it satisfies the power object property.

{\em Step 1}: It actually suffices to show that $\mathbf{P}A$ is in $\mathbf{SCan}$. Because then, by Proposition \ref{SCan has limits} and Corollary \ref{SCan pres and refl limits}, $A \times \mathbf{P}A$ is in $\mathbf{SCan}$ (and is such a product in both $\mathbf{N}$ and $\mathbf{SCan}$), so that by Proposition \ref{SCan closed under monos} and fullness of $\mathbf{SCan}$, $(\iota_A^{-1} \times \id_{\mathbf{P}A}) \circ m_{\subseteq^{\mathbf{T}}_A} : \hspace{1pt} \subseteq^{\mathbf{T}}_A \hspace{1pt} \rightarrowtail A \times \mathbf{P}A$ is in $\mathbf{SCan}$.

$\mathbf{PT}A$ along with $m_{\subseteq^{\mathbf{T}}_{\mathbf{T}A}} : \hspace{1pt} \subseteq^{\mathbf{T}}_{\mathbf{T}A} \hspace{1pt} \rightarrow \mathbf{TT}A \times \mathbf{PT}A$ is a power object of $\mathbf{TT}A$ in $\mathbf{N}$. So $\mathbf{PT}A$ along with $(\iota_{\mathbf{T}A}^{-1} \times \id_{\mathbf{PT}A}) \circ m_{\subseteq^{\mathbf{T}}_{\mathbf{T}A}} : \hspace{1pt} \subseteq^{\mathbf{T}}_{\mathbf{T}A} \hspace{1pt} \rightarrow \mathbf{T}A \times \mathbf{PT}A$ is a power object of $\mathbf{T}A$ in $\mathbf{N}$. Moreover, $\mathbf{P}A$ is a power object of $\mathbf{T}A$ in $\mathbf{N}$. Therefore, using the natural isomorphism $\mu : \mathbf{P}\mathbf{T} \rightarrow \mathbf{T}\mathbf{P}$, we obtain an isomorphism $\alpha : \mathbf{P}A \xrightarrow{\sim} \mathbf{P}\mathbf{T}A \xrightarrow{\sim} \mathbf{T}\mathbf{P}A$ in $\mathbf{N}$. This results in the following two-way pullback in both $\mathbf{M}$ and $\mathbf{N}$:
\[
\begin{tikzcd}[ampersand replacement=\&, column sep=small]
{\subseteq^{\mathbf{T}}} \ar[rr, leftrightarrow, "{\sim}"] \ar[d, rightarrowtail] \&\& {\subseteq^{\mathbf{T}}} \ar[d, rightarrowtail] \\
\mathbf{T} A \times \mathbf{P}A \ar[rr, bend left, "{\sim}", "{\id \times \alpha}"'] \&\& \mathbf{T} A \times \mathbf{T}\mathbf{P} A \ar[ll, bend left, "{\sim}"', "{\id \times \alpha^{-1}}"]
\end{tikzcd}
\]
Since this pullback-square can be filled with $\iota_{\mathbf{P}A}$ in place of $\alpha$, the uniqueness property of the pullback implies that $\iota_{\mathbf{P}A} = \alpha$, whence $\iota_{\mathbf{P}A}$ is in $\mathbf{N}$ and $\mathbf{P}A$ is in $\mathbf{SCan}$.

{\em Step 2}: Let $r : R \rightarrowtail A \times B$ be a mono in $\mathbf{SCan}$. By Corollary \ref{SCan mono iff N mono}, $r$ is also monic in $\mathbf{N}$. So since $\iota_A$ is an isomorphism in $\mathbf{N}$, there is a unique $\chi$ in $\mathbf{N}$ such that this is a pullback in $\mathbf{N}$:
\[
\begin{tikzcd}[ampersand replacement=\&, column sep=small]
R \ar[rrr] \ar[d, rightarrowtail, "{r}"'] \&\&\& \subseteq^{\mathbf{T}}_A \ar[d, rightarrowtail, "{(\iota^{-1}_A \times \id_{\mathbf{P}A}) \circ m_{\subseteq^{\mathbf{T}}_A}}"] \\
A \times B \ar[rrr, "{\id_A \times \chi}"'] \&\&\& A \times \mathbf{P} A 
\end{tikzcd} 
\]
By Step 1, by fullness and by Corollary \ref{SCan pres and refl limits}, it is also a pullback in $\mathbf{SCan}$. To see uniqueness of $\chi$ in $\mathbf{SCan}$, suppose that $\chi\hspace{1pt}'$ were some morphism in $\mathbf{SCan}$ making this a pullback in $\mathbf{SCan}$ (in place of $\chi$). Then by Corollary \ref{SCan pres and refl limits}, it would also make it a pullback in $\mathbf{N}$, whence $\chi = \chi\hspace{1pt}'$.
\end{proof}

\begin{thm}
$\mathbf{SCan}_{( \mathbf{M}, \mathbf{N})}$ is a topos.
\end{thm}
\begin{proof}
A category with finite limits and power objects is a topos.
\end{proof}

\section{Where to go from here?}\label{where to go NF cat}

$\mathrm{(I)ML(U)}_\mathsf{Cat}$ has been shown, respectively, to interpret $\mathrm{(I)NF(U)}_\mathsf{Set}$, and has conversely been shown to be interpretable in $\mathrm{(I)ML(U)}_\mathsf{Class}$, thus yielding equiconsistency results. Since the axioms of a Heyting category can be obtained from the axioms of topos theory, it is natural to ask:

\begin{que}
Can the axioms of $\mathrm{(I)ML(U)}_\mathsf{Cat}$ be simplified? In particular, is it necessary to include the axioms of Heyting categories or do these follow from the other axioms?
\end{que}

To be able to interpret the set theory in the categorical semantics, this research introduces the axiomatization $\mathrm{(I)ML(U)}_\mathsf{Cat}$ corresponding to predicative $\mathrm{(I)ML(U)}_\mathrm{Class}$. This is analogous to the categories of classes for conventional set theory studied e.g. in \cite{ABSS14}. But it remains to answer:

\begin{que}
How should the speculative theory $\mathrm{(I)NF(U)}_\mathsf{Cat}$ naturally be axiomatized? I.e. what is the natural generalization of $\mathrm{(I)NF(U)}_\mathsf{Set}$ to category theory, analogous to topos theory as the natural generalization of conventional set theory? Moreover, can any category modeling this theory be canonically extended to a model of $\mathrm{(I)ML(U)}_\mathsf{Cat}$, or under what conditions are such extensions possible?
\end{que}

Closely intertwined with this question, is the potential project of generalizing to topos theory the techniques of automorphisms and self-embeddings of non-standard models of set theory. In particular, the endofunctor $\mathbf{T}$ considered in this research should arise from an automorphism or self-embedding of a topos. This would be a natural approach to constructing a rich variety of categories modeling the speculative theory $\mathrm{(I)NF(U)}_\mathsf{Cat}$, many of which would presumably be extensible to models of $\mathrm{(I)ML(U)}_\mathsf{Cat}$.

\end{document}